\input ./style/arxiv-general.cfg
\documentclass[aap,MSNbibl,secthm,dvips]{arximspdf}
\makeatletter
   \@ifpackageloaded{graphicx}{}{\usepackage{graphicx}}
\makeatother
\usepackage{mathbh}

\doi{10.1214/14-AAP1037} 
\volume{25}
\issue{4}
\pubyear{2015}
\firstpage{1827}
\lastpage{1867}
\docsubty{FLA}

\makeatletter
\newcommand{\ds}{\displaystyle}
\newproclaim{nota}{Notation}
\newtheorem{cor}[thm]{Corollary}
\newtheorem{prop}[thm]{Proposition}
\newtheorem{lem}[thm]{Lemma}
\newproclaim{df}[thm]{Definition}
\newproclaim{rem}[thm]{Remark}
\newproclaim{ex}[thm]{Example}
\makeatother

\begin{document}
\begin{frontmatter}

\title{Strict local martingales and bubbles}
\runtitle{Strict local martingales and bubbles}

\begin{aug}
\author[A]{\fnms{Constantinos}~\snm{Kardaras}\ead[label=e1]{k.kardaras@lse.ac.uk}},
\author[B]{\fnms{D\"orte}~\snm{Kreher}\thanksref{T1}\ead[label=e2]{kreher@math.hu-berlin.de}}\break
\and
\author[C]{\fnms{Ashkan}~\snm{Nikeghbali}\corref{}\ead[label=e3]{ashkan.nikeghbali@math.uzh.ch}}
\runauthor{C. Kardaras, D. Kreher and A. Nikeghbali}
\affiliation{London School of Economics  and Political Science,\break Humboldt-Universit\"at zu Berlin and Universit\"at Z\"urich}
\address[A]{C. Kardaras\\
Department of Statistics\\
London School of Economics\\
\quad and Political Science\\
London WC2A 2AE\\
United Kingdom\\
\printead{e1}} 
\address[B]{D. Kreher\\
Department of Mathematics\\
Humboldt-Universit\"at zu Berlin\\ Berlin\\
Germany\\
\printead{e2}}
\address[C]{A. Nikeghbali\\ Institut f\"ur Mathematik \& Institut\\
\quad f\"ur Banking and
Finance\\ Universit\"at Z\"urich\\ Z\"urich\\
Switzerland\\
\printead{e3}}
\end{aug}
\thankstext{T1}{Supported by SNF Grant 137652 during the preparation
of the paper.}

\received{\smonth{12} \syear{2013}}

%
\begin{abstract}
This paper deals with asset price bubbles modeled by strict local
martingales. With any strict local martingale, one can associate a new
measure, which is studied in detail in the first part of the paper. In
the second part, we determine the ``default term'' apparent in
risk-neutral option prices if the underlying stock exhibits a bubble
modeled by a strict local martingale. Results for certain path
dependent options and last passage time formulas are given.
\end{abstract}

\end{frontmatter}

\section{Introduction}

The goal of this paper is to determine the influence of asset price
bubbles on the pricing of derivatives. Asset price bubbles have been
studied extensively in the economic literature looking for explanations
of why they arise, but have only recently gained attention in
mathematical finance by Cox and Hobson \cite{CoxHobson},
Pal and Protter \cite{PalP},
and Jarrow et al. \cite{JPScomp,JPffbubbles,JPSincomp}. When an asset price bubble
exists, the market price of the asset is higher than its fundamental
value. From a mathematical point of view, this is the case when the
stock price process is modeled by a positive strict local martingale
under the equivalent local martingale measure. Here, by a strict local
martingale, we understand a local martingale, which is not a true
martingale. Strict local martingales were first studied in the context
of financial mathematics by  Delbaen and Schachermayer \cite{DSbessel}.
Afterward, Elworthy et al. \cite{ELYtails,ELYradial} studied some of their
properties including their tail behaviour. More recently, the interest
in them grew again (cf., e.g., Mijatovic and Urusov \cite{Urusov}) because of
their importance in the modelling of financial bubbles.

Obviously, there are options for which well-known results regarding
their valuation in an arbitrage-free market hold true without
modification, regardless of whether the underlying is a strict local
martingale or a true martingale under the risk-neutral measure. One
example is the put option with strike $K\geq0$. If the underlying is
modeled by a continuous local martingale $X$ with $X_0=1$, it is shown
by Madan et al. \cite{MRY} that the risk-neutral value of the
put option can be expressed in terms of the last passage time of the
local martingale $X$ at level $K$ via
\[
\mathbb{E}(K-X_T)^+=\mathbb{E} \bigl((K-X_\infty)^+
\mathbh{1}_{\{
\rho_K^X\leq T\}} \bigr) \qquad \mbox{with } \rho_K^X=
\sup\{t\geq0| X_t=K\}.
\]

This formula does not require $X$ to be a true martingale, but is also
valid for strict local martingales. However, if we go from puts to
calls, the strict locality of $X$ is relevant. The general idea is to
reduce the call case to the put case by a change of measure with
Radon--Nikodym density process given by $(X_t)_{t\geq0}$ as done in
Madan et al. \cite{MRY} in the case where $X$ is a true
martingale. However, if $X$ is a strict local martingale, this does not
define a measure any more. Instead, we first have to localize the
strict local martingale and can thus only define measures on stopped
sub-$\sigma$-algebras. Under certain conditions on the probability
space, we can then extend the so-defined consistent family of measures
to a measure defined on some larger $\sigma$-field. Under the new
measure, the reciprocal of $X$ turns into a~true martingale. The
conditions we impose are taken from F{\"o}llmer \cite{exit}, who
requires the filtration to be a standard system (cf. Definition~\ref{defss}). This way we get an extension of Theorem~4 in
Delbaen and Schachermayer \cite{DSbessel} to general probability spaces and c\`adl\`ag local
martingales. We study the behavior of $X$ and other local martingales
under the new measure.

Using these technical results, we obtain decomposition formulas for
some classes of European path-dependent options under the NFLVR
condition. These formulas are extensions of Proposition~2 in Pal and Protter \cite{PalP},
which deals with nonpath-dependent options. We
decompose the option value into a difference of two positive terms, of
which the second one shows the influence of the stock price bubble.

Furthermore, we express the risk-neutral price of an exchange option in
the presence of asset price bubbles as an expectation involving the
last passage time at the strike level under the new measure. This
result is similar to the formula for call options derived by Madan,
Roynette and Yor \cite{fromto} or Yen and Yor \cite{YenYor} for the
case of reciprocal Bessel processes. We can further generalize their
formula to the case where the candidate density process for the
risk-neutral measure is only a~strict local martingale. Then the NFLVR
condition is not fulfilled and risk-neutral valuation fails, so that we
have to work under the real-world measure. Since in this case the price
of a zero coupon bond is decreasing in maturity even with an interest
rate of zero, some people refer to this as a bond price bubble as
opposed to the stock price bubbles discussed above; see, for
example, Hulley \cite{Hulley}. In this general setup, we obtain expressions
for the option value of European and American call options in terms of
the last passage time and the explosion time of the deflated price
process, which make some anomalies of the prices of call options in the
presence of bubbles evident: European calls are not increasing in
maturity any longer and the American call option premium is not equal
to zero any more; see, for example, Cox and Hobson \cite{CoxHobson}.

This paper is organized as follows: In the next section, we study
strictly positive (strict) local martingales in more detail. On the one
hand, we demonstrate ways of how one can obtain strict local
martingales, while on the other hand we construct the above mentioned
measure associated with a c\`adl\`ag strictly positive local martingale
on a general filtered probability space with a standard system as
filtration. We give some examples of this construction in Section~\ref{examples}. In Section~\ref{app1}, we then apply our results to the
study of asset price bubbles. After formally defining the financial
market model, we obtain decomposition formulas for certain classes of
European path-dependent options, which show the influence of stock
price bubbles on the value of the options under the NFLVR condition. In
Section~\ref{PQ}, we further study the relationship between the
original and the new measure constructed in Section~\ref{newmeasure},
which we apply in Section~\ref{app2} to obtain last passage time
formulas for the European and American exchange option in the presence
of asset price bubbles. Moreover, we show how this result can be
applied to the real-world pricing of European and American call
options. The last section contains some results about multivariate
strict local martingales.

\setcounter{footnote}{1}
\section{C\`adl\`ag strictly positive strict local martingales}

When dealing with continuous strictly positive strict local
martingales, a very useful tool is the result from~\cite{DSbessel}; see
also Proposition~6 in \cite{PalP}, which states that every such process
defined as the coordinate process on the canonical space of
trajectories can be obtained as the reciprocal of a ``Doob
$h$-transform''\footnote{Note that we abuse the word ``Doob
$h$-transform'' in this context slightly, since Doob $h$-transforms are
normally only defined in the theory of Markov processes.} with $h(x)=x$
of a continuous nonnegative true martingale. Conversely, any such
transformation of a continuous nonnegative martingale, which hits zero
with positive probability, yields a strict local martingale.

The goal of this section is to extend these results to c\`adl\`ag
processes and general probability spaces satisfying some extra
conditions, which were introduced in \cite{para} and used in a similar
context in \cite{exit}. While the construction of strict local
martingales from true martingales follows from an application of the
Lenglart--Girsanov theorem, the converse theorem relies as in
\cite{DSbessel} on the construction of the F\"ollmer exit measure of a
strictly positive local martingale as done in \cite{exit} and
\cite{Meyer}.

\subsection{How to obtain strictly positive strict local martingales}

Examples of continuous strict local martingales have been known for a
long time; the canonical example being the reciprocal of a Bessel
process of dimension 3. This example can be generalized to a broader
class of transient diffusions, which taken in natural scale turn out to
be strict local martingales; see, for example, \cite{ELYtails}. A
natural way to construct strictly positive continuous strict local
martingales is given in Theorem~1 of \cite{DSbessel}. There it is shown
that every uniformly integrable nonnegative martingale with positive
probability to hit zero gives rise to a change of measure such that its
reciprocal is a strict local martingale under the new measure. For the
noncontinuous case and for not necessarily uniformly integrable
martingales, we now give a simple extension of the just mentioned
theorem from \cite{DSbessel}.

%
\begin{thm}\label{thm1}
Let $(\Omega,\mathcal{F},(\mathcal{F}_t)_{t\geq0},\mathsf{Q})$ be
the natural augmentation
of some filtered probability space with $\mathcal{F}=\bigvee_{t\geq
0}\mathcal{F}_t$,
that is, the filtration $({\mathcal{F}}_t)$ is right-continuous and
${\mathcal{F}}_0$
contains all ${\mathcal{F}}_t$-negligible sets for all $t\geq0$. Let
$Y$ be a
nonnegative $\mathsf{Q}$-martingale starting from $Y_0=1$. Set $\tau
=\inf\{t\geq0\dvtx Y_t=0\}$ and assume that $\mathsf{Q}(\tau<\infty)>0$.
Furthermore, suppose that $Y$ does not jump to zero $\mathsf{Q}$-almost
surely. For all $t\geq0$, define a probability measure $\mathsf{P}_t$
on $\mathcal{F}
_t$ via $\mathsf{P}_t=Y_t.\mathsf{Q}|_{\mathcal{F}_t}$; in
particular, $\mathsf{P}_t\ll\mathsf{Q}|_{{\mathcal{F}
}_t}$. Assume that either $Y$ is uniformly integrable under ${\mathsf
{Q}}$ or
that the nonaugmented probability space satisfies condition
(P).\footnote{Condition $(P)$ first appeared in \cite{para} and was
later used in \cite{newkind}. We recall its definition in the \hyperref[AppP]{Appendix}.} Then we can extend the consistent family
$(\mathsf{P}_t)_{t\geq0}$ to a measure $\mathsf{P}$ on the augmented space
$(\Omega,\mathcal{F}
,(\mathcal{F}_t)_{t\geq0})$. Under the measure $\mathsf{P}$ the process $Y$ does never reach
zero and
its reciprocal $1/Y$ is a strict local $\mathsf{P}$-martingale.
\end{thm}

\begin{pf}
Since the underlying probability space satisfies the natural
assumptions, we may choose a c\`adl\`ag version of $Y$; see, for
example, Propositions 3.1~and~3.3 in \cite{newkind}. Especially, this
means that $\tau$ is a well-defined stopping time. If $Y$ is a
uniformly integrable martingale, the measure $\mathsf{P}$ is defined
on $\mathcal{F}$
by $d\mathsf{P}=Y_\infty \,d\mathsf{Q}$. In the other case, when the
probability space
fulfills condition $(P)$, the existence of the measure $\mathsf{P}$ follows
from Corollary~4.9 in \cite{newkind}.
Moreover, note that
\[
\mathsf{P}(\tau<\infty)=\lim_{t\rightarrow\infty}\mathsf{P}(\tau \leq t)=\lim
_{t\rightarrow
\infty}\mathbb{E}^\mathsf{Q} (\mathbh{1}_{\{\tau\leq t\}
}Y_t
)=0,
\]
therefore, the process $1/Y$ is a $\mathsf{P}$-almost surely well-defined
semi-martingale.
The result now follows from Corollary~3.10 in Chapter III of \cite{JaSh} applied to $M'_t:=\frac{1}{Y_t}\mathbh{1}_{\{\tau>t\}}$, once
we can
show that $(M'_{t\wedge\tau_n} Y_{t\wedge\tau_n})$ with $\tau
_n=\inf\{t\geq0\dvtx Y_t\leq\frac{1}{n}\}$ is a local $\mathsf{Q}$-martingale
for every $n\in\mathbb{N}$. But,
\[
M'_{t\wedge\tau_n} Y_{t\wedge\tau_n}=\mathbh{1}_{\{\tau>t\wedge
\tau_n\}
}=1\qquad
\mathsf{Q}\mbox{-a.s.},
\]
because $Y$ does not jump to zero $\mathsf{Q}$-almost surely. This trivially
proves the martingale property. Finally, the strictness of the local
martingale $1/Y$ under $\mathsf{P}$ follows from
\[
\mathbb{E}^\mathsf{P} \biggl(\frac{1}{Y_t} \biggr)=\mathsf{Q}(
\tau>t)<1
\]
for $t$ large enough, since by assumption $\mathsf{Q} (\tau
<\infty
 )>0$.
\end{pf}

Starting with a Brownian motion stopped at zero under $\mathsf{Q}$, it
is easy
to show that the associated strict local martingale under $\mathsf{P}$
is the
reciprocal of the three-dimensional Bessel process, which is the
canonical example of a strict local martingale (cf. Example~1 in \cite{PalP}). Without stating the general result, the above construction is
also applied in \cite{Chybi} to construct examples of strict local
martingales with jumps related to Dunkl Markov processes on the one
hand (cf. Proposition~3 in \cite{Chybi}) and semi-stable Markov
processes on the other hand (cf. Proposition~5 in~\cite{Chybi}). Apart
from the previous, there do not seem to be any well-known examples of
strict local martingales with jumps. Note, however, that one can
construct an example by taking any continuous strict local martingale
and multiplying it with the stochastic exponential of an independent
compound Poisson process or any other independent and strictly positive
jump martingale.

In the following example, we construct a ``nontrivial'' positive
strict local martingale with jumps by a shrinkage of filtration.

%
\begin{ex}
Consider the well-known reciprocal three-dimensional Bessel process $Y$
as a function of a three-dimensional standard Brownian motion
$B=(B^1,B^2,B^3)$ starting from $B_0=(1,0,0)$, that is,
\[
Y = \frac{1}{\sqrt{(B^1)^2+(B^2)^2+(B^3)^2}}.
\]
We define the filtrations $(\mathcal{F}_t)_{t\geq0}$ and $(\mathcal
{G}_t)_{t\geq0}$
through $\mathcal{F}_t=\sigma(B^1_s,B^2_s,B^3_s; s\leq t)$ and
$\mathcal{G}_t=\sigma
(B^1_s,B^2_s;s\leq t)$, as well as the filtration $(\mathcal{H}_t)_{t\geq0}$ through
\[
\mathcal{H}_t=\mathcal{F}_{{\lfloor nt\rfloor}/{n}}\vee \mathcal{G}_t=
\sigma \biggl(B^1_s,B^2_s,s\leq
t; B^3_u,u\leq\frac{\lfloor nt\rfloor}{n} \biggr)
\]
for some $n\in\mathbb{N}$. It is shown in Theorem~15 of \cite{FollmerProtter} that not only $Y$ itself is a strict local $(\mathcal{F}_t)_{t
\geq0}$-martingale, but that also the optional projection of $Y$ onto
$(\mathcal{G}_t)_{t \ge0}$ is a continuous local $(\mathcal{G}_t)_{t
\geq
0}$-martingale. Since $\mathcal{G}_t\subset\mathcal{H}_t\subset
\mathcal{F}_t$ for $t \geq0$,
it follows by Corollary~2 of \cite{FollmerProtter} that then the
optional projection of $Y$ onto $(\mathcal{H}_t)_{t\geq0}$, denoted by~$^{\circ}Y$, is also a local martingale. However, since its
expectation process is decreasing, $^{\circ}Y$ must be a strict local
martingale that jumps at $t\in\frac{\mathbb{N}}{n}$. Indeed, since
$B^3$ is
a~Brownian motion independent of $B^1$ and $B^2$, $B_t^3$ given
$\mathcal{H}_t$
is normally distributed with mean $B^3_{{\lfloor nt\rfloor}/{n}}$
and variance $t-\frac{\lfloor nt\rfloor}{n}$. Therefore,
$^\circ Y$ is given by the explicit formula $^\circ
Y_t=u(B^1_t,B^2_t,B^3_{{\lfloor nt\rfloor}/{n}},t)$,
where
\begin{eqnarray*}
u(x,y,a,t)
&= &  \int_\mathbb{R}  \bigl(x^2+y^2+z^2
\bigr)^{-1/2}\\
&&\hspace*{10pt}{}\times\sqrt{\frac
{1}{2\pi
 (t-{\lfloor nt\rfloor}/{n} )}}\exp \biggl(-
\frac
{1}{2 (t-{\lfloor nt\rfloor}/{n} )}(z-a)^2 \biggr)\,dz.
\end{eqnarray*}
\end{ex}

%
\begin{rem}
In the recent preprint \cite{Protterjumps}, the method of filtration
shrinkage is applied in greater generality to construct more
sophisticated examples of strict local martingales with jumps.
\end{rem}

%
\begin{ex}
As a further example, any nonnegative nonuniformly integrable
$(\mathcal{F}
_t)_{t\geq0}$-martingale $Z$ with $Z_0=1$ allows to construct a
strictly positive strict local martingale $Y$ relative to a new
filtration $(\tilde{\mathcal{F}}_t)_{t\geq0}$ through a
deterministic change
of time: simply set
\[
Y_t= %
\cases{\ds\frac{1}{2} (1+Z_{{t}/({1-t})} ), &\quad $0
\leq t<1$,
\vspace*{3pt}\cr
\ds\frac{1}{2} \Bigl(1+\lim_{t\rightarrow\infty}Z_t
\Bigr), & \quad $1\leq t$}
\]
and define $\tilde{\mathcal{F}}_t=\mathcal{F}_{{t}/({1-t})}$ for
$t<1$ and $\tilde
{\mathcal{F}}_t=\mathcal{F}_\infty$ for $t\geq1$. Since $Z$ is
\textit{not}
uniformly integrable, we have $\mathbb{E}Y_1<Y_0=Z_0=1$ almost surely. Note,
however, that $Y$ is a~true martingale on the interval $[0,1)$. Instead
of setting $Y$ constant for $t\geq1$ one can also define $Y$ to behave
like any other strictly positive local martingale starting from
$Y_1:=\frac{1}{2}(1+\lim_{t\rightarrow\infty}Z_t)$ on $[1,\infty)$.
\end{ex}

\subsection{From strictly positive strict local martingales to true
martingales}\label{newmeasure}

In the following, let $(\Omega,\mathcal{F},(\tilde{\mathcal
{F}}_t)_{t\geq0},\mathsf{P})$ be
a filtered probability space. Furthermore, we denote by $({\mathcal{F}
}_t)_{t\geq0}$ the right-continuous augmentation of $(\tilde{\mathcal{F}
}_t)_{t\geq0}$, that is, ${\mathcal{F}}_t:=\tilde{\mathcal
{F}}_{t+}=\bigcap_{s>t}\tilde{\mathcal{F}}_s$ for all $t\geq0$. Note, however, that the
filtration is \textit{not} completed with the negligible sets of
$\mathcal{F}$.

%
\begin{df}[(cf. \cite{exit})]\label{defss}
 Let $T$ be a partially ordered nonvoid index
set and let $(\tilde{\mathcal{F}}_t)_{t\in T}$ be a filtration on
$\Omega$.
Then $(\tilde{\mathcal{F}}_t)_{t\in T}$ is called a standard system if:
\begin{itemize}
\item each measurable space $(\Omega,\tilde{\mathcal{F}}_t)$ is a standard
Borel space, that is, $\tilde{\mathcal{F}}_t$ is $\sigma$-isomorphic
to the
$\sigma$-field of Borel sets on some complete separable metric space;
\item for any increasing sequence $(t_i)_{i\in\mathbb{N}}\subset T$
and for
any $A_1\supset A_2\supset\cdots\supset A_i\supset\cdots\,$, where
$A_i$ is an atom of $\tilde{\mathcal{F}}_{t_i}$, we have $\bigcap_i
A_i\neq
\varnothing$.
\end{itemize}
\end{df}

As noted in \cite{newkind}, the filtration $\tilde{\mathcal
{F}}_t=\sigma
(X_s,s\leq t)$, where $X_t(\omega)=\omega(t)$ is the coordinate
process on the space $C(\mathbb{R}_+,\mathbb{R}_+)$ of nonexplosive
nonnegative
continuous functions, is not a standard system. However, it will be
seen below that when dealing with strict local martingales it is
natural to work on the space of all $\overline{\mathbb{R}}_+=\mathbb
{R}_+\cup\{\infty\}
$-valued processes that are continuous up to some time $\alpha\in
[0,\infty]$ and constant afterward. As noted in example (6.3) in \cite{exit}, the filtration generated by the coordinate process on this
space is indeed a standard system. More generally, we have the
following lemma.

%
\begin{lem}\label{canonical}
Let $\Omega=D'(\mathbb{R}_+,\overline{\mathbb{R}}_+^n)$ be the
space of functions from
$\mathbb{R}_+$ into $\overline{\mathbb{R}}_+^n$ with componentwise
right-continuous paths
$(\omega_i(t))_{t\geq0}, i=1,\ldots,n$, that
have left limits on $(0,\alpha(\omega))$ for some \mbox{$\alpha
(\omega)\in[0,\infty]$} and remain constant on $[\alpha(\omega
),\infty)$ at the value $\lim_{t\uparrow\alpha(\omega)}\omega
_i(t)$ if this limit exists and at $\infty$ otherwise.
We denote by $(X_t)_{t\geq0}$ the coordinate\vspace*{1pt} process, that
is, $X_t(\omega_1,\ldots,\omega_n)=(\omega_1(t),\ldots,\omega
_n(t))$, and by $(\tilde{\mathcal{F}}_t)_{t\geq0}$ the canonical filtration
generated by the coordinate process, that is, $\tilde{\mathcal
{F}}_t=\sigma
(X_s; s\leq t)$.
Furthermore, set $\mathcal{F}=\bigvee_{t\geq0}\tilde{\mathcal
{F}}_t$. Then, $(\tilde
{\mathcal{F}}_t)_{t\geq0}$ is a standard system on the space $(\Omega
,\mathcal{F}
,(\tilde{\mathcal{F}}_t)_{t\geq0})$.
The same is true, if we replace $D'(\mathbb{R}_+,\overline{\mathbb
{R}}_+^n)$ by its
subspace $C'(\mathbb{R}_+,\overline{\mathbb{R}}_+^n)$ of functions
which are componentwise
continuous on some $(0,\alpha(\omega))$ and remain constant on
$[\alpha(\omega),\infty)$ at the value $\lim_{t\uparrow\alpha
(\omega)}\omega_i(t)$ if this limit exists and at $\infty$ otherwise.
\end{lem}

\begin{pf}
We prove the claim for $\Omega=D'(\mathbb{R}_+,\overline{\mathbb
{R}}_+^n)$. The case
$\Omega=\break C'(\mathbb{R}_+,\overline{\mathbb{R}}_+^n)$ is done in a
similar way.
As in \cite{dell}, we define a bijective mapping $i$ from~$\Omega$ to
some subspace $A\subset(\overline{\mathbb{R}}_+^n)^\mathbb{Q}$
(where here $\mathbb{Q}$ denotes
the set of all rational numbers), via $\omega\mapsto(X_r(\omega
))_{r\in\mathbb{Q}}$. It is clear that $i$ is bijective and we have
$\mathcal{F}
=i^{-1}(\mathcal{B}(A))$.
Furthermore, a sequence $A_1\supset A_2\supset\cdots\supset A_i\supset
\cdots$ of atoms of $\mathcal{F}_{t_i}=\sigma(X_s; s\leq t_i)$
defines a
component-wise c\`adl\`ag function on the interval $[0,\lim{t_i}]\cap
[0,\alpha(\omega))$, which is constant on $[0,\lim t_i]\cap[\alpha
(\omega),\infty)$, for every increasing sequence $(t_i)_{i\in\mathbb{N}
}\subset\mathbb{R}_+$. This function can easily be extended to an
element of
$D'(\mathbb{R}_+,\overline{\mathbb{R}}_+^n)$.
\end{pf}

Recall that for any $(\mathcal{F}_t)_{t \geq0}$-stopping time $\tau$ the
sigma-algebra $\mathcal{F}_{\tau-}$ is defined as
\[
\mathcal{F}_{\tau-}=\sigma\bigl(\tilde{\mathcal{F}}_0, \bigl\{\{\tau >t\}\cap\Gamma\dvtx \Gamma\in\mathcal{F}_t, t>0 \bigr\}\bigr).
\]

%
\begin{lem}[(cf. \cite{exit}, Remark~6.1)] \label{standard}
Let $(\tilde{\mathcal{F}}_t)_{t\geq0}$
be a standard system on~$\Omega$. Then for any increasing sequence
$(\tau_n)_{n\in\mathbb{N}}$ of $({\mathcal{F}}_t)$-stopping times
the family $(\mathcal{F}
_{\tau_{n-}})_{n\in\mathbb{N}}$ is also a standard system.
\end{lem}

\begin{nota*}
When working on the subspace $(\Omega,\mathcal{F}_{\tau-})$ of
$(\Omega,\mathcal{F}
)$, where $\tau$ is some $(\mathcal{F}_t)$-stopping time, we must
restrict the
filtration to $(\mathcal{F}_{t\wedge\tau-})_{t\geq0}$, where with a slight
abuse of notation we set $\mathcal{F}_{t\wedge\tau-}:=\mathcal
{F}_t\cap\mathcal{F}_{\tau-}$.
In the following, we may also write $(\mathcal{F}_t)_{0\leq t<\tau}$
for the
filtration on $(\Omega,\mathcal{F}_{\tau-},\mathsf{P})$.
\end{nota*}

Working with standard systems will allow us to derive for every
strictly positive strict local $\mathsf{P}$-martingale the existence
of a
measure $\mathsf{Q}$ on $ (\Omega,{\mathcal{F}}_{\tau
-},\break ({\mathcal{F}}_{t})_{0\leq
t<\tau} )$, such that the reciprocal of the strict local
$\mathsf{P}$-martingale is a true $\mathsf{Q}$-martingale. In Section~\ref{app1}, we
will use this result to reduce calculations involving strict local
martingales to the much easier case of true martingales.

From Theorem~4 in \cite{DSbessel} and Proposition~6 in \cite{PalP},
we know that every continuous local martingale understood as the
canonical process on $C(\mathbb{R}_+,\overline{\mathbb{R}}_+)$ gives
rise to a new measure
under which its reciprocal turns into a true martingale.
In the context of arbitrage theory, similar results have recently been
derived and applied by \cite{FernKar} and \cite{Ruf} for continuous
processes in a Markovian setting.
Theorem~\ref{komplett} below is an extension of these results to more
general probability spaces and c\`adl\`ag processes.
Its proof relies on the construction of the F\"ollmer measure
(cf. \cite{exit} and \cite{Meyer}); nevertheless, we will give a
detailed proof, since it is essential for the rest of the paper.

%
\begin{prop}\label{dom}
Let $ (\Omega,\mathcal{F},(\tilde{\mathcal{F}}_t)_{t\geq
0},\mathsf{P} )$ be a
filtered probability space and assume that $(\tilde{\mathcal
{F}}_t)_{t\geq0}$
is a standard system.
Let $X$ be a c\`adl\`ag local martingale on the space $(\Omega
,\mathcal{F}
,({\mathcal{F}}_{t})_{t\geq0},\mathsf{P})$ with values in $(0,\infty
)$ and $X_0=1$
$\mathsf{P}$-almost surely.
We define $\tau^X_n:=\inf\{t\geq0\dvtx X_t>n\}\wedge n$ and $\tau
^X=\lim_{n\rightarrow\infty}\tau^X_n$.
Then there exists a unique probability measure ${\mathsf{Q}}$ on
$
(\Omega,{\mathcal{F}}_{\tau^X-},({\mathcal{F}}_{t\wedge\tau
^X-})_{t\geq0}
)$, such that $\frac{d{\mathsf{P}}}{d{\mathsf{Q}}}|_{{\mathcal
{F}}_t\cap{\mathcal{F}}_{\tau
^X-}}=\frac{1}{X_t}\mathbh{1}_{\{t<\tau^X\}}$ for all $t\geq0$. Moreover,
$1/X$ is a local $\mathsf{Q}$-martingale on the interval $[0,\tau^X)$ which
does not jump to zero $\mathsf{Q}$-almost surely.
\end{prop}

\begin{pf}
First, note that $\tau^X_n$ is an $(\mathcal{F}_t)_{t\geq
0}$-stopping time
and the process $(X_{t\wedge\tau^X_n})_{t\geq0}$ is a uniformly
integrable $ \{(\mathcal{F}_t)_{t\geq0},\mathsf{P} \}
$-martingale for all
$n\in\mathbb{N}$. Indeed, if $(\sigma_m)$ is any localizing sequence
for $X$
such that $\mathbb{E}^\mathsf{P}X_{\sigma_m}=1$ for all $m\in
\mathbb{N}$, then
\[
X_{\tau_n^X\wedge\sigma_m}\leq n\vee X_{\tau_n^X} \quad \mbox{and}  \quad \mathbb{E}
^\mathsf{P}(n\vee X_{\tau_n^X})\leq n+\mathbb{E}^\mathsf{P}X_{\tau
_n^X}
\leq n+1
\]
by the super-martingale property of $X$. By the dominated convergence
theorem, we thus conclude that $\mathbb{E}^\mathsf{P}X_{\tau
_n^X}=1$, and thus
$(\tau_n^X)$ is a localizing sequence as well.

Furthermore, $\mathsf{P}(\tau^X=\infty)=1$, since a positive c\`adl\`ag
local martingale does not explode almost surely. We define on $(\Omega
,{\mathcal{F}}_{\tau^X_n})$ the probability measure $\tilde{\mathsf
{Q}}_n$ via
$\tilde{\mathsf{Q}}_n=X_{\tau^X_n}\cdot {\mathsf{P}}|_{{\mathcal
{F}}_{\tau^X_n}}$ for all $n\in
\mathbb{N}$. The family $(\tilde{\mathsf{Q}}_n)_{n\in\mathbb{N}}$
constitutes a consistent
family of probability measures on $(\mathcal{F}_{\tau^X_n})_{n\geq
1}$: If
$A\in{\mathcal{F}}_{\tau^X_n}$, then
\[
\tilde{\mathsf{Q}}_{n+k}(A)=\mathbb{E}^{{\mathsf{P}}}
(X_{\tau
^X_{n+k}}\mathbh{1}_A )=\mathbb{E}^{{\mathsf{P}}}(X_{\tau^X_n}
\mathbh{1}_A)=\tilde {\mathsf{Q}}_n(A),
\]
that is, $\tilde{\mathsf{Q}}_{n+k}|_{{\mathcal{F}}_{\tau
^X_n}}=\tilde{\mathsf{Q}}_n$ for
all $n,k\in\mathbb{N}$. This induces a sequence of consistently defined
measures $(\mathsf{Q}_n)_{n\in\mathbb{N}}$ on the sequence
$(\mathcal{F}_{\tau^X_n-})_{n\in
\mathbb{N}}$, which is a standard system by Lemma~\ref{standard}.
Note that
$\mathcal{F}_{\tau^X-}=\bigvee_{n\geq1}\mathcal{F}_{\tau^X_n-}$,
since\vspace*{1pt} $(\tau
^X_n)_{n\geq1}$ is increasing. 
We can thus apply Theorem~3.2 together with Theorem~4.1 in Chapter V of
\cite{para} (cf. also Theorem~6.2 in \cite{exit}), which yield the
existence of a unique measure $\mathsf{Q}$ on $ (\Omega,\mathcal
{F}_{\tau
^X-},(\mathcal{F}_{t\wedge\tau^X-})_{t\geq0} )$ such that
$\mathsf{Q}|_{\mathcal{F}
_{\tau^X_n-}}=\mathsf{Q}_n=\tilde{\mathsf{Q}}_n|_{\mathcal
{F}_{\tau^X_n-}}$. Moreover,
since $\{\tau^X_n<\tau^X_m\}\in\mathcal{F}_{\tau^X_m-}$,
\begin{eqnarray*}
&& \mathsf{Q}\bigl(\tau^X_n<\tau^X\bigr)
\\
&& \qquad =\lim
_{m\rightarrow\infty}\mathsf {Q}\bigl(\tau^X_n<\tau
^X_m\bigr)=\lim_{m\rightarrow\infty}\tilde{
\mathsf{Q}}_m\bigl(\tau^X_n<\tau
^X_m\bigr)= \lim_{m\rightarrow\infty}
\mathbb{E}^{{\mathsf{P}}} (\mathbh{1}_{\{
\tau^X_n<\tau^X_m\}}X_{\tau
^X_m} )
\\
&&\qquad = \lim_{m\rightarrow\infty}\mathbb{E}^{{\mathsf{P}}} (\mathbh
{1}_{\{\tau^X_n<\tau^X_m\}
}X_{\tau^X_n} )= \mathbb{E}^{{\mathsf{P}}} (
\mathbh{1}_{\{\tau^X_n<\tau^X\}
}X_{\tau^X_n} )= \mathbb{E}^{{\mathsf{P}}}
(X_{\tau^X_n} )=1,
\end{eqnarray*}
that is, $1/X$ does not jump to zero under $\mathsf{Q}$. Therefore, if
$\Lambda_n\in\mathcal{F}_{\tau^X_n}$, then
\begin{eqnarray*}
\mathsf{Q}(\Lambda_n)&=&\mathsf{Q} \bigl(\Lambda_n\cap
\bigl\{\tau ^X>\tau^X_n\bigr\} \bigr)= \lim
_{m\rightarrow\infty}\mathsf{Q} \bigl(\Lambda_n\cap\bigl\{\tau
^X_m>\tau^X_n\bigr\} \bigr)
\\
&=& \lim_{m\rightarrow\infty}\mathbb{E}^{{\mathsf
{P}}} (X_{\tau^X_m}
\mathbh{1} _{\Lambda_n}\mathbh{1}_{\{\tau^X_m>\tau^X_n\}} ) =\lim_{m\rightarrow\infty}
\mathbb{E}^{{\mathsf{P}}} (X_{\tau
^X_n}\mathbh{1}_{\Lambda_n}\mathbh{1}
_{\{\tau^X_m>\tau^X_n\}} )
\\
&=& \mathbb{E}^{{\mathsf{P}}} (X_{\tau^X_n}\mathbh{1}_{\Lambda
_n} )=
\tilde{\mathsf{Q} }_n(\Lambda_n).
\end{eqnarray*}
Therefore, $\mathsf{Q}|_{\mathcal{F}_{\tau^X_n}}=\tilde{\mathsf
{Q}}_n$ for all $n\in\mathbb{N}$.

Now let $S$ be an $({\mathcal{F}}_t)_{t\geq0}$-stopping time. Note
that $\{
S<\tau^X_n\}\in{\mathcal{F}}_S$ and $\{S<\tau^X_n\}\in{\mathcal
{F}}_{\tau^X_n}$. Thus,
\begin{eqnarray*}
{\mathsf{Q}}\bigl(S<\tau^X_n\bigr)&=&
\tilde{\mathsf{Q}}_n\bigl(S<\tau ^X_n\bigr)=
\mathbb{E}^{{\mathsf{P}}} (\mathbh{1} _{\{S<\tau^X_n\}}X_{\tau^X_n} )=
\mathbb{E}^{{\mathsf
{P}}} \bigl(\mathbh{1}_{\{S<\tau
^X_n\}}\mathbb{E}^{{\mathsf{P}}}(X_{\tau^X_n}|
\mathcal{F}_S) \bigr)
\\
&=&\mathbb{E}^{{\mathsf{P}}} (\mathbh{1} _{\{S<\tau^X_n\}}X_S ).
\end{eqnarray*}
Since $\mathsf{P}(\tau^X_n<\tau^X=\infty) = 1$, taking the limit as
$n\rightarrow
\infty$ in the above equation yields
%
%
\begin{equation}
\label{chara}
{\mathsf{Q}}\bigl(S<\tau^X\bigr)=\mathbb{E}^{{\mathsf{P}}}
(\mathbh {1}_{\{S<\infty\}}X_S ).
\end{equation}
Applied to the stopping time $S_A:=S\mathbh{1}_A+\infty\mathbh
{1}_{A^c}$, where $A\in
\mathcal{F}_S$, this gives
\[
{\mathsf{Q}}\bigl(S<\tau^X,A\bigr)=\mathbb{E}^{{\mathsf{P}}} (
\mathbh {1}_{A\cap\{S<\infty\}}X_S ).
\]
Especially, if $S$ is finite ${\mathsf{P}}$-almost surely, then
${\mathsf{Q}}(S<\tau
^X,A)=\mathbb{E}^{{\mathsf{P}}}(X_S\mathbh{1}_A)$ for $A\in
{\mathcal{F}}_S$.
If $A\in{\mathcal{F}}_t\cap{\mathcal{F}}_{\tau^X-}$, then
\begin{eqnarray*}
{\mathsf{P}}(A)&=&\lim_{n\rightarrow\infty}{\mathsf{P}}\bigl(A\cap\bigl\{ t<
\tau^X_n\bigr\}\bigr)=\lim_{n\rightarrow
\infty}
\mathbb{E}^{{\mathsf{Q}}} \biggl(\mathbh{1}_A\mathbh{1}_{\{
t<\tau^X_n\}}
\frac{1}{X_{\tau
^X_n}} \biggr)
\\
&=& \lim_{n\rightarrow\infty}\mathbb{E}^{{\mathsf{Q}}} \biggl(\mathbh
{1}_A\mathbh{1}_{\{t<\tau^X_n\}}\frac
{1}{X_t} \biggr)=
\mathbb{E}^{{\mathsf{Q}}} \biggl(\mathbh{1}_A\mathbh
{1}_{\{t<\tau^X\}}\frac
{1}{X_t} \biggr).
\end{eqnarray*}
Therefore, $\frac{d{\mathsf{P}}}{d{\mathsf{Q}}}|_{\mathcal{F}_t\cap
\mathcal{F}_{\tau
^X-}}=\frac{1}{X_t}\mathbh{1}_{\{t<\tau^X\}}$ for all $t\geq0$.

Finally, note that because $(X_t^{\tau^X_n})_{t\geq0}$ is a strictly
positive uniformly integrable $\mathsf{P}$-martingale for all $n\in
\mathbb{N}$,
${\mathsf{P}}|_{\mathcal{F}_{\tau^X_n}}\sim{\mathsf{Q}}|_{\mathcal
{F}_{\tau^X_n}}$ and
\[
d{\mathsf{P}}|_{\mathcal{F}_{\tau^X_n}}=\frac{1}{X_{\tau
^X_n}}d{\mathsf{Q}}|_{\mathcal{F}_{\tau
^X_n}}
\quad \Leftrightarrow \quad \frac{d{\mathsf{Q}}}{d{\mathsf
{P}}}\bigg|_{{\mathcal{F}}_{t\wedge\tau
^X_n}}={X_{t\wedge\tau^X_n}}\qquad \forall  t
\geq0.
\]
Thus,
\[
\mathbb{E}^{{\mathsf{Q}}} \biggl(\frac{1}{X_{t\wedge\tau
^X_n}}\Big|\mathcal{F} _s
\biggr)= \mathbb{E}^{{\mathsf{P}}} \biggl(\frac{1}{X_{t\wedge\tau
^X_n}}\cdot
\frac{X_{t\wedge\tau^X_n}}{X_{s\wedge\tau^X_n}} \Big|\mathcal{F}_s \biggr)=\frac{1}{X_{s\wedge\tau^X_n}}
\]
for $s\leq t$, that is, $\frac{1}{X}$ is a local $\mathsf{Q}$-martingale on
the interval $\bigcup_{n\in\mathbb{N}}[0,\tau_n^X]=[0,\tau^X)$.
\end{pf}

%
\begin{cor}\label{cor}
Under the assumptions of Proposition~\ref{dom}, $X$ is a strict local
$\mathsf{P}$-martingale, if and only if \mbox{$\mathsf{Q}(\tau
^X<\infty)>0$.}
\end{cor}

\begin{pf}
It follows directly from equation (\ref{chara}) that $\mathsf
{Q}(t<\tau
^X)=\mathbb{E}^\mathsf{P}X_t$, which is smaller than 1 for some $t$,
iff $X$ is a
strict local martingale under $\mathsf{P}$.
\end{pf}

%
\begin{rem}
Corollary~\ref{cor} makes clear why we cannot work with the natural
augmentation of $(\tilde{\mathcal{F}}_t)_{t\geq0}$. Indeed, we have
$A_n:=\{\tau^X \leq n\}\in\mathcal{F}_n\cap\mathcal{F}_{\tau^X-}$ and
${\mathsf{P}}(A_n)=0$ for
all $n\in\mathbb{N}$, while ${\mathsf{Q}}(A_n)>0$ for some $n$ if
$X$ is a strict
local $\mathsf{P}$-martingale. However, it is in general rather inconvenient
to work without any augmentation, especially if one works with an
uncountable number of stochastic processes. For this reason, a new kind
of augmentation---called the $(\tau^X_n)$-natural augmentation---is
introduced in \cite{KN}, which is suitable for the change of measure
from $\mathsf{P}$ to $\mathsf{Q}$ undertaken here. Since for the financial
applications in the second part of this paper the setup introduced
above is already sufficient, we do not bother about this augmentation
here and refer the interested reader to \cite{KN} for more technical details.
\end{rem}

\phantomsection
\label{extension}

In the following, we extend the measure $\mathsf{Q}$ in an arbitrary
way from
$\mathcal{F}_{\tau^X-}$ to $\mathcal{F}_\infty=\bigvee_{t\geq
0}\tilde{\mathcal{F}}_t$. For
notational convenience, we assume that $\mathcal{F}=\mathcal
{F}_\infty$. In fact, it
is always possible to extend a probability measure from $\mathcal
{F}_{\tau
^X-}$ to $\mathcal{F}$: since $(\Omega,\tilde{\mathcal{F}}_t)$ is a
standard Borel
space for every $t\geq0$ and $(\Omega,\mathcal{F}_{\tau^X_n-})$ is a
standard Borel space for all $n\in\mathbb{N}$ by Lemma~\ref{standard}, it
follows from Theorem~4.1 in \cite{para} that $(\Omega,\mathcal{F})$ and
$(\Omega,\mathcal{F}_{\tau^X-})$ are also standard Borel spaces. Especially,
they are countably generated which allows us to apply Theorem~3.1 of
\cite{Ershov} that guarantees an extension of $\mathsf{Q}$ from
$\mathcal{F}_{\tau
^X-}$ to $\mathcal{F}$. Moreover, it does not matter for the results
how we
extend it, because all events that happen with positive probability
under $\mathsf{P}$ take place before time $\tau^X$ under $\mathsf
{Q}$ almost surely.
However, if $Y$ is any process on $(\Omega,\mathcal{F},(\mathcal
{F}_t)_{t\geq0},\mathsf{P})$,
then $Y_t$ is only defined on $\{t<\tau^X\}$ under $\mathsf{Q}$. Especially,
if $Y$ is a $\mathsf{P}$-semi-martingale, then $Y^{\tau^X_n}$ is a
$\mathsf{Q}
$-semi-martingale for each $n\in\mathbb{N}$ as follows from Girsanov's
theorem, since $\mathsf{Q}|_{\mathcal{F}_{\tau_n^X}} \sim\mathsf
{P}|_{\mathcal{F}_{\tau_n^X}}$.
Therefore, $Y$ is a $\mathsf{Q}$-semi-martingale on the stochastic interval
$\bigcup_{n\in\mathbb{N}} [0,\tau_n^X]$ or a ``semi-martingale up
to time
$\tau^X$'' in the terminology of \cite{Jacod}. We note that in
general it may not be possible to extend $Y$ to the whole positive real
line under $\mathsf{Q}$ in such a way that $Y$ remains a~semi-martingale.
Indeed, according to Proposition~5.8 of \cite{Jacod} such an extension
is possible if and only if $Y_{\tau^X-}$ exists in $\mathbb{R}_+$
$\mathsf{Q}$-almost surely. We define the process~$\tilde{Y}$ as
%
%
\begin{equation}
\label{tilde}
\tilde{Y}_t=
\cases{\ds Y_t, &\quad  $t<
\tau^X$,
\vspace*{3pt}\cr
\ds\liminf_{s\rightarrow\tau^X,s<\tau^X,s\in\mathbb{Q}}Y_s, &\quad  $\tau
^X\leq t<\infty$.}
\end{equation}
Note that $\tilde{Y}_t=Y_t$ on $\{t<\tau^X\}$. The above\vspace*{1pt} definition
specifies an extension of the process $Y$, which is a priori only
defined up to time $\tau^X$, to the whole positive real line. In the
following, we will work with this extension.

%
\begin{lem}
Under the assumptions of Proposition~\ref{dom}, we have $\frac
{1}{\tilde{X}_t}=\frac{1}{X_t}\mathbh{1}_{\{t<\tau^X\}}$.
Furthermore, the
process $ (\frac{1}{\tilde{X}_t} )_{t\geq0}$ is a true
$\mathsf{Q}$-martingale for any extension of $\mathsf{Q}$ from
$\mathcal{F}_{\tau^X-}$ to
$\mathcal{F}$.
\end{lem}

\begin{pf}
First, note that $\mathsf{Q}$-almost surely
\begin{eqnarray*}
\limsup_{n\rightarrow\infty} \frac{1}{X_{t\wedge\tau
^X_n}}&=&\limsup
_{n\rightarrow\infty} \biggl(\frac{1}{X_t}\mathbh{1}_{\{t<\tau
^X_n\}}+
\frac
{1}{X_{\tau^X_n}}\mathbh{1}_{\{t\geq\tau^X_n\}} \biggr)
\\
&\leq& \frac{1}{X_t}\mathbh{1}_{\{t<\tau^X\}}+\limsup_{n\rightarrow\infty
}
\frac{1}{n}\mathbh{1} _{\{t\geq\tau^X_n\}}=\frac{1}{X_t}
\mathbh{1}_{\{t<\tau^X\}}
\end{eqnarray*}
and
\begin{eqnarray*}
\liminf_{n\rightarrow\infty} \frac{1}{X_{t\wedge\tau^X_n}} &=&  \liminf
_{n\rightarrow\infty} \biggl(\frac{1}{X_t}\mathbh{1}_{\{t<\tau
^X_n\}}+
\frac
{1}{X_{\tau^X_n}}\mathbh{1}_{\{t\geq\tau^X_n\}} \biggr)\\
& \geq &  \liminf
_{n\rightarrow
\infty} \frac{1}{X_t}\mathbh{1}_{\{t<\tau^X_n\}}=
\frac
{1}{X_t}\mathbh{1}_{\{t<\tau
^X\}}.
\end{eqnarray*}
Thus, $\frac{1}{\tilde{X}_t}=\frac{1}{X_t}\mathbh{1}_{\{t<\tau^X\}
}$. Furthermore,
\begin{eqnarray*}
0&\leq&\frac{1}{X_{\tau^X-}}\mathbh{1}_{\{\tau^X<\infty\}}=\lim_{k\rightarrow
\infty}
\frac{1}{X_{\tau^X-}}\mathbh{1}_{\{\tau^X<k\}}= \lim_{k\rightarrow\infty}\lim
_{n\rightarrow\infty}\frac
{1}{X_{\tau^X_n}}\mathbh{1}_{\{
\tau^X<k\}}
\\
&\leq&\lim_{k\rightarrow\infty}\lim_{n\rightarrow\infty}
\frac
{1}{n}\mathbh{1}_{\{\tau
^X<k\}}=0
\end{eqnarray*}
implies that $X_{\tau^X-}=\infty$ on $\{\tau^X<\infty\}$ $\mathsf{Q}
$-almost surely. From the proof of Proposition~\ref{dom}, we know that
$\frac{1}{X^{\tau_n^X}}$ is a true $\mathsf{Q}$-martingale for all
$n\in\mathbb{N}
$. By the definition of $\tau_n^X$, we have for any integer $n\geq t$
\[
X_{t\wedge\tau^X_n}=\tilde{X}_{t\wedge\tau^X_n}=\tilde {X}_{t\wedge\inf\{s\geq0\dvtx \tilde{X}_s>n\}}\geq
\tilde{X}_t\wedge 1 \quad \Rightarrow \quad  \frac{1}{X_{t\wedge\tau^X_n}}\leq
\frac{1}{\tilde{X}_t\wedge
1}=1\vee\frac{1}{\tilde{X}_t}.
\]
Because
\[
\mathbb{E}^{\mathsf{Q}} \biggl(\frac{1}{\tilde{X}_t} \biggr)=\mathbb
{E}^{\mathsf{Q}} \biggl(\liminf_{n\rightarrow\infty}\frac{1}{{X}_{t\wedge\tau_n^X}}
\biggr)\leq \liminf_{n\rightarrow\infty}\mathbb{E}^{\mathsf{Q}} \biggl(
\frac
{1}{{X}_{t\wedge\tau_n^X}} \biggr)=1,
\]
the dominated convergence theorem implies that for all $0\leq s\leq t$
\begin{eqnarray*}
\mathbb{E}^{{\mathsf{Q}}} \biggl(\frac{1}{\tilde{X}_t}\Big|{\mathcal
{F}}_s \biggr)& =& \mathbb{E}^{{\mathsf{Q}}} \biggl(\lim
_{n\rightarrow
\infty}\frac{1}{{X}_{t\wedge
\tau^X_n}}\Big|{\mathcal{F}}_s \biggr)
=
\lim_{n\rightarrow\infty
}\mathbb{E}^{{\mathsf{Q}
}} \biggl(\frac{1}{X_{t\wedge\tau^X_n}}\Big|{
\mathcal{F}}_s \biggr) \\
&=& \lim_{n\rightarrow\infty}
\frac{1}{X_{s\wedge\tau^X_n}}=\frac
{1}{\tilde{X}_s}.
\end{eqnarray*}
\upqed\end{pf}

To simplify notation, we identity in the following the process $X$ with
$\tilde{X}$. We summarize our results so far in the following theorem.

%
\begin{thm}\label{komplett}
Let $ (\Omega,\mathcal{F},(\tilde{\mathcal{F}}_t)_{t\geq
0},\mathsf{P} )$ be a
filtered probability space and assume that $(\tilde{\mathcal
{F}}_t)_{t\geq0}$
is a standard system.
Let $X$ be a c\`adl\`ag local martingale on $(\Omega,\mathcal
{F},({\mathcal{F}
}_{t})_{t\geq0},\mathsf{P})$ with values in $(0,\infty)$ and $X_0=1$
$\mathsf{P}$-almost surely.
We define $\tau^X_n:=\inf\{t\geq0\dvtx X_t>n\}\wedge n$  and
$\tau^X=\lim_{n\rightarrow\infty}\tau^X_n$.
Then there exists a probability measure ${\mathsf{Q}}$ on $(\Omega,\mathcal{F}
_\infty )$ such that $1/X$ is a $\mathsf{Q}$-martingale, which
does not
jump to zero $\mathsf{Q}$-almost surely, and such that
$\mathsf{Q}(A,\tau^X>t)=\mathbb{E}^\mathsf{P}(X_t\mathbh{1}_{A})$
for all $t\geq0$ and $A\in\mathcal{F}
_t$. In particular, $\mathsf{P}|_{\mathcal{F}_t}\ll\mathsf
{Q}|_{\mathcal{F}_t}$ for all $t\geq0$.
\end{thm}

Note that in the case where $X$ is a strict local $\mathsf{P}$-martingale
Theorem~\ref{komplett} is a~precise converse to Theorem~\ref{thm1}, if one
identifies $X$ of Theorem~\ref{komplett} with $1/Y$ of Theorem~\ref{thm1}.

\section{Examples}\label{examples}

In this section, we shed new light on some known examples of strict
local martingales by applying the theory from the last section for illustration.

\subsection{Continuous local martingales}

For the following examples, we work on the path space $C'(\mathbb
{R}_+,\overline{\mathbb{R}
}_+)$ with $W$ denoting the coordinate process. Here, $(\mathcal
{F}_t)_{t\geq
0}$ is the right-continuous augmentation of the canonical filtration
generated by the coordinate process and $\mathsf{P}$ is Wiener measure.

\subsubsection{Exponential local martingales}

Suppose that $X$ has dynamics
\[
dX_t=X_t b(Y_t)\,dW_t,\qquad
X_0=1,
\]
where $Y$ is assumed to be a (possibly explosive) diffusion with
\[
dY_t=\mu(Y_t)\,dt+\sigma(Y_t)\,dW_t,\qquad
Y_0=y\in\mathbb{R}.
\]
Here, $b(\cdot),\mu(\cdot)$ and $\sigma(\cdot)$ are chosen such
that both SDEs allow for strong solutions and guarantee $X$ to be
strictly positive. Exponential local martingales of this type are
further studied in \cite{Urusov}. Under $\mathsf{Q}$ the dynamics of
$\frac
{1}{X}$ up to time $\tau^X$ are
\[
d \biggl(\frac{1}{X_t} \biggr)=-\frac{b(Y_t)}{X_t}\,dW_t^\mathsf{Q}
\]
for a $\mathsf{Q}$-Brownian motion $W^\mathsf{Q}$ defined up to time
$\tau^X$, and
the $\mathsf{Q}$-dynamics of $Y_t$ up to time $\tau^X$ are
\[
dY_t= \bigl[\mu(Y_t)+\sigma(Y_t)b(Y_t)
\bigr]\,dt+\sigma (Y_t)\,dW_t^\mathsf{Q}.
\]
Notably, the criterion whether $X$ is a strict local or a true $\mathsf{P}$-martingale from \cite{Urusov}, Theorem~2.1, is deterministic and
only involves the functions $b,\sigma$ and $\mu$ via the scale
function of the original diffusion $Y$ under $\mathsf{P}$ and an auxiliary
diffusion $\tilde{Y}$, whose dynamics are identical with the $\mathsf{Q}$-dynamics of $Y$ stated above.

\subsubsection{Diffusions in natural scale}\label{difex}

We now take $X$ to be a local $\mathsf{P}$-martingale of the form
\[
dX_t=\sigma(X_t)\,dW_t,\qquad  X_0=1,
\]
assuming that $\sigma(x)$ is locally bounded and bounded away from
zero for $x>0$ and $\sigma(0)=0$. Using the results from \cite{DelbaenShirakawa}, we know that $X$ is strictly positive, whenever
\[
\int_0^1\frac{x}{\sigma^2(x)}\,dx=\infty,
\]
which we shall assume in the following. Furthermore, $X$ is a strict
local martingale, if and only if
\[
\int_1^\infty\frac{x}{\sigma^2(x)}\,dx<\infty.
\]
We know that $\frac{1}{X}$ is a $\mathsf{Q}$-martingale, where $\frac
{d\mathsf{P}}{d\mathsf{Q}}|_{\mathcal{F}_t}=\frac{1}{X_t}$, with
decomposition
\[
d \biggl(\frac{1}{X_t} \biggr)=-\frac{\sigma(X_t)}{X_t^2}\,dW_t^\mathsf
{Q}=\overline {\sigma} \biggl(\frac{1}{X_t} \biggr)\,dW_t^\mathsf{Q}
\]
for a $\mathsf{Q}$-Brownian motion $W^{\mathsf{Q}}$ defined up to
time $\tau^X$ and
$\overline{\sigma}(y):=-y^2\cdot\sigma (\frac{1}{y} )$.
Note that
\[
\int_1^\infty\frac{y}{\overline{\sigma}^2(y)}\,dy=\int
_0^1\frac
{x}{\sigma^2(x)}\,dx=\infty,
\]
which confirms that $\frac{1}{X}$ is a true $\mathsf{Q}$-martingale.
We see
that, if $X$ is a strict local martingale under $\mathsf{P}$, then
\[
\int_0^1\frac{y}{\overline{\sigma}^2(y)}\,dy=\int
_1^\infty\frac
{x}{\sigma^2(x)}\,dx<\infty,
\]
that is, $\frac{1}{X}$ hits zero in finite time $\mathsf{Q}$-almost surely.

\subsection{Jump example}\footnote{This example is taken from \cite{Chybi}.
However, we corrected a small mistake concerning the time-scaling.}
Let $\Omega=D'(\mathbb{R}_+,\overline{\mathbb{R}})$ with $(\xi
_t)_{t\geq0}$ denoting
the coordinate process and $(\mathcal{F}_t)_{t\geq0}$ being the
right-continuous augmentation of the canonical filtration generated by
the coordinate process. Assume that under $\mathsf{P}$, $(\xi
_t)_{t\geq0}$
is a one-dimensional L\'evy process with $\xi_0=0$, $\mathbb
{E}^\mathsf{P}\exp(b\xi
_t)=\exp(t\rho(b))<\infty$ for all $t\geq0$ and characteristic exponent
\[
\Psi(\lambda)=ia\lambda+\frac{1}{2}\sigma^2
\lambda^2+\int_\mathbb{R} \bigl(1-e^{i\lambda x}+i
\lambda x\mathbh{1}_{\{|x|<1\}} \bigr)\pi(dx),
\]
where $a\in\mathbb{R}$, $\sigma^2\geq0$ and $\pi$ is a positive
measure on
$\mathbb{R}\setminus\{0\}$ such that $\int(1\wedge|x|^2)\pi
(dx)<\infty$.
Define
\[
X_t=Y_t^b\exp \biggl(-\rho(b)\int
_0^t\frac{ds}{Y_s} \biggr),
\]
where $(Y_t)_{t\geq0}$ is a semi-stable Markov process, that
is, $ (\frac{1}{c}Y_{ct}^{(x)} )_{t\geq0}\stackrel
{(d)}{=} (Y_t^{(xc^{-1})} )_{t\geq0}$ for all $c>0$,
implicitly defined via
\[
\exp(\xi_t)=Y_{\int_0^t\exp(\xi_s)\,ds}.
\]
Following \cite{Chybi}, $(X_t)_{t\geq0}$ is a positive strict local
martingale if $a$ and $b$ satisfy
\begin{eqnarray*}
-a+\int_{|x|>1} x\pi(dx) &\geq & 0, \\
-a+b\sigma^2-\int
_{|x|<1}x\bigl(1-e^{bx}\bigr)\pi(dx)+\int
_{|x|>1}xe^{bx}\pi(dx) &<& 0.
\end{eqnarray*}
Furthermore, under the new measure $\mathsf{Q}$ the process
\[
\frac{1}{X_t}=Y_t^{-b}\exp \biggl(\rho(b)\int
_0^t\frac
{ds}{Y_s} \biggr)
\]
is a true martingale, where now $(\xi_t)_{t\geq0}$ has characteristic
exponent $\tilde{\Psi}$ with
\[
\tilde{\Psi}(u)=\Psi(u-ib)-\Psi(-ib).
\]

\section{Application to financial bubbles I: Decomposition
formulas}\label{app1}

In this section, we apply our results to option pricing in the presence
of strict local martingales. For this, we assume that the following
standing assumption (S) holds throughout the entire
section:

\begin{longlist}[(S)]
\item[(S)] \textit{$X$ is assumed to be a c\`adl\`ag strictly
positive local martingale on $(\Omega,\mathcal{F}, (\mathcal
{F}_t)_{t\geq0},\mathsf{P})$,
whose filtration is the right-continuous augmentation of a standard
system and $\mathcal{F}=\bigvee_{t\geq0}\mathcal{F}_t$. We assume
that $X_0=1$ and
set $\tau_n^X=\inf\{t\geq0| X_t>n\}\wedge n$ for all $n\in\mathbb
{N}$ and
$\tau^X=\lim_{n\rightarrow\infty}\tau_n^X$. Furthermore, we denote
by $\mathsf{Q}$ any extension to $(\Omega,\mathcal{F})$ of the measure associated
with $X$,
defined in Theorem~\ref{komplett}.}
\end{longlist}

We consider a financial market model which satisfies the NFLVR property
as defined in \cite{NFLVRlocal}. We denote by $\mathsf{P}$ an equivalent
local martingale measure (ELMM). Assuming that the interest rate equals
zero, we interpret $X$ as the (discounted) stock price process, which
is a local martingale under $\mathsf{P}$. In this context, the
question of
whether $X$ is a strict local or a true $\mathsf{P}$-martingale determines
whether there exists a stock price bubble. If $X$ is a strict local
$\mathsf{P}$-martingale, the fundamental value of the asset (given by the
conditional expectation) deviates from its actual market price $X$.
Several authors (cf., e.g., \cite{CoxHobson,JPScomp,JPSincomp,PalP})
have interpreted this as the existence of a stock price bubble, which
we formally define as follows.

%
\begin{df}
With the previous notation, the \textit{asset price bubble} for the
stock price process $X$ between time $t\geq0$ and time $T\geq t$ is
equal to the $\mathcal{F}_t$-measurable random\vspace*{-3pt} variable
\[
\gamma_X(t,T):=X_t-\mathbb{E}^\mathsf{P}(X_T|
\mathcal{F}_t).
\]
\end{df}

%
\begin{rem}
For $t=0$, we recover the `default' function $\gamma
_X(0,T)=X_0-\mathbb{E}^\mathsf{P}
X_T$ of the local martingale $X$, which was introduced in \cite{ELYradial}. Here, the term `default' refers to the locality property
of $X$ and measures its failure of being a~martingale. In \cite{ELYtails,ELYradial}, the authors derive several expressions for the
default function in terms of the first hitting time, the local time and
the last passage time of the local\vspace*{-2pt} martingale.
\end{rem}

%
\begin{rem}
Note that the above definition of a bubble depends on the measure
$\mathsf{P}$, which may be viewed as the subjective valuation measure of
a certain economic agent. From the agent's point of view, the asset
price contains a bubble. Only in a complete market, that is, if and
only if $\mathsf{P}$ is the unique ELMM, the notion of a bubble becomes
universal without any element of\vspace*{-2pt} subjectivity.
\end{rem}

In Proposition~7 of \cite{PalP}, the price of a nonpath-dependent
option written on a stock, whose price process is a (strict) local
martingale, is decomposed into a ``normal'' (``nonbubble'') term and a
default term. In the following, we give an extension of this theorem to
a certain class of path-dependent options. For this, let us introduce
the following notation for all\vspace*{-2pt} $k\in\mathbb{N}$:
\[
\mathbb{R}^k_+=\bigl\{x\in\mathbb{R}^k\dvtx
x_l\geq0, l=1,\ldots,k\bigr\}, \qquad \mathbb{R}^k_{++}=
\bigl\{x\in\mathbb{R} ^k\dvtx x_l> 0, l=1,\ldots,k\bigr\}.
\]

%
\begin{thm}\label{inbetween}
Let $0\leq t_1<t_2<\cdots<t_n<\infty$ and consider a Borel-measurable
nonnegative function $h\dvtx\mathbb{R}_{++}^n\rightarrow\mathbb{R}_+$.
Define the function
$g(x):=x_n\cdot h (\frac{1}{x_1},\ldots,\frac{1}{x_n} )$
for all $x= (x_1,\ldots,x_n )\in\mathbb{R}_{++}^n$.\vspace*{-2pt} Then
\[
\mathbb{E}^\mathsf{P}h (X_{t_1},\ldots,X_{t_n} )=
\mathbb {E}^\mathsf{Q} \biggl(g \biggl(\frac{1}{X_{t_1}},\ldots,
\frac{1}{X_{t_n}} \biggr)\mathbh{1}_{\{
\tau^X>t_n\}
} \biggr).
\]
Now\vspace*{-2pt} suppose that the following limits exist in $\mathbb{R}_+$ for
$y_i\in\mathbb{R}
_{++}, i=1,\ldots,\break n-1$:
\begin{eqnarray*}
\lim_{|z|\rightarrow0} g (y_1,\ldots,y_k;z_1,
\ldots,z_{n-k} )&=:&\eta_k(y_1,
\ldots,y_k),\qquad  k=1,\ldots,n-1,
\\[-4pt]
\lim_{|z|\rightarrow0} g (z_1,\ldots,z_{n} )&=:&
\eta_0.
\end{eqnarray*}
Define $\overline{g}\dvtx A\rightarrow\mathbb{R}_+$ as the extension of
$g$ from $\mathbb{R}
^n_{++}$ to $A\subset\mathbb{R}^n_+$, where $A$ is defined as $A:=\{
x\in\mathbb{R}
^n_+\dvtx \mbox{if }x_k=0\mbox{ for some }k=1,\ldots,n,
\mbox{then }x_l=0 \  \forall  l\geq k\}$.\vspace*{-3pt} Then
%
\begin{eqnarray}\label{d1}
&& \mathbb{E}^\mathsf{P}h (X_{t_1},
\ldots,X_{t_n} )
\nonumber\\[-8pt]\\[-8pt]\nonumber
&&\qquad =\mathbb {E}^\mathsf{Q}\overline {g} \biggl(
\frac{1}{X_{t_1}},\ldots,\frac{1}{X_{t_n}} \biggr)-\sum
_{k=0}^{n-1} \mathbb{E}^\mathsf{Q} \bigl(
\mathbh{1}_{\{t_k<\tau
^X\leq t_{k+1}\}}\cdot\eta _k\bigl(X^k\bigr)
\bigr),
\end{eqnarray}
where we set $t_0=0$ and $X^k= (\frac{1}{X_{t_1}},\ldots,\frac
{1}{X_{t_k}} )$ for $k=1,\ldots,n-1$, $X^0\equiv0$.

In particular, if $\eta_k(\cdot)\equiv c_k, k=1,\ldots,n-1$, are
constant, then
%
\begin{equation}
\label{d2}
\hspace*{7pt}\mathbb{E}^\mathsf{P}h (X_{t_1},
\ldots,X_{t_n} )= \mathbb {E}^\mathsf{Q}\overline {g} \biggl(
\frac{1}{X_{t_1}},\ldots,\frac{1}{X_{t_n}} \biggr)-\sum
_{k=0}^{n-1} c_k\cdot\mathsf{Q}
\bigl(t_k<\tau^X\leq t_{k+1} \bigr).
\end{equation}
\end{thm}

\begin{pf}
First, note that
\[
\mathbh{1}_{\{\tau^X>t_n\}}=\mathbh{1}_{\{\tau^X>t_1\}}\mathbh {1}_{\{\tau^X>t_2\}}
\cdots\mathbh{1} _{\{\tau^X>t_{n-1}\}}\mathbh{1}_{\{\tau^X>t_n\}}.
\]
Using the change of measure $d\mathsf{P}|_{\mathcal{F}_{t_n}}=\frac
{1}{X_{t_n}}\,d\mathsf{Q}
|_{\mathcal{F}_{t_n}}$ on $\{\tau^X>t_n\}$, we deduce
\begin{eqnarray*}
&& \mathbb{E}^\mathsf{P}h(X)\\
&&\qquad= \mathbb{E}^\mathsf{Q} \biggl(g \biggl(
\frac
{1}{X} \biggr)\mathbh{1}_{\{\tau
^X>t_n\}} \biggr)\\
&&\qquad=
\mathbb{E}^\mathsf{Q} \biggl(g \biggl(\frac
{1}{X} \biggr)
\mathbh{1}_{\{
\tau^X>t_1\}}\cdots\mathbh{1}_{\{\tau^X>t_n\}} \biggr)\\
&&\qquad=
\mathbb{E}^\mathsf{Q} \biggl(\mathbh{1}_{\{\tau^X>t_1\}
}
\mathbb{E}^\mathsf{Q} \biggl(\mathbh{1}_{\{\tau^X>t_2\}}\cdots\\
&&\qquad\quad {}\times\mathbb{E}^\mathsf{Q} \biggl(\mathbh{1}_{\{\tau
^X>t_{n-1}\}}
\mathbb{E}^\mathsf{Q} \biggl(\mathbh{1}_{\{\tau^X>t_n\}
}g \biggl(
\frac
{1}{X} \biggr)\Big|\mathcal{F}_{t_{n-1}} \biggr)\Big|\mathcal
{F}_{t_{n-2}} \biggr)\cdots\Big|\mathcal{F}_{t_2} \biggr)\Big|\mathcal
{F}_{t_1} \biggr)\bigg).
\end{eqnarray*}
Because on $\{\tau^X>t_{n-1}\}$, we have
\begin{eqnarray*}
&& \mathbb{E}^\mathsf{Q} \biggl(\mathbh{1}_{\{\tau^X>t_n\}}g \biggl(
\frac{1}{X} \biggr)\Big|\mathcal{F}_{t_{n-1}} \biggr)\\
&&\qquad=
\mathbb{E}^\mathsf{Q} \biggl(\overline{g} \biggl(\frac{1}{X}
\biggr)\Big|\mathcal{F}_{t_{n-1}} \biggr)-\mathbb{E}^\mathsf{Q} \bigl(
\mathbh{1}_{\{t_{n-1}<\tau^X\leq t_n\}} \eta_{n-1} \bigl(X^{n-1} \bigr)|
\mathcal{F}_{t_{n-1}} \bigr),
\end{eqnarray*}
it follows that
\begin{eqnarray*}
\mathbb{E}^\mathsf{P}h(X)&=& \mathbb{E}^\mathsf{Q}\biggl(
\mathbh{1}_{\{
\tau^X>t_1\}}\mathbb{E}^\mathsf{Q} \biggl(\mathbh{1}
_{\{\tau^X>t_2\}}\cdots\\
&& {}\times\mathbb{E}^\mathsf{Q} \biggl(\mathbh{1}_{\{
\tau^X>t_{n-2}\}}
\mathbb{E}^\mathsf{Q} \biggl(\mathbh{1}_{\{\tau^X>t_{n-1}\}}\overline{g} \biggl(
\frac
{1}{X} \biggr)\Big|\mathcal{F}_{t_{n-2}} \biggr)\cdots\Big|\mathcal
{F}_{t_1} \biggr) \biggr)
\\
&&{}- \mathbb{E}^\mathsf{Q} \bigl(\mathbh{1}_{\{\tau^X>t_1\}
}
\mathbh{1}_{\{\tau^X>t_2\}}\cdots\mathbh{1}_{\{
\tau^X>t_{n-1}\}}\mathbh{1}_{\{t_{n-1}<\tau^X\leq t_n\}}
\eta _{n-1} \bigl(X^{n-1} \bigr) \bigr).
\end{eqnarray*}
Similarly, on $\{\tau^X>t_{n-2}\}$ we have
\begin{eqnarray*}
&& \mathbb{E}^\mathsf{Q} \biggl(\mathbh{1}_{\{\tau^X>t_{n-1}\}
}\overline{g}
\biggl(\frac
{1}{X} \biggr)\Big|\mathcal{F}_{t_{n-2}} \biggr)\\
&& \qquad =
\mathbb{E}^\mathsf {Q} \biggl(\overline {g} \biggl(\frac{1}{X}
\biggr)\Big|\mathcal{F}_{t_{n-2}} \biggr)-\mathbb {E}^\mathsf{Q} \bigl(
\mathbh{1}_{\{t_{n-2}<\tau^X\leq t_{n-1}\}}\eta_{n-2} \bigl(X^{n-2} \bigr)|
\mathcal{F}_{t_{n-2}} \bigr),
\end{eqnarray*}
and we deduce that
\begin{eqnarray*}
\mathbb{E}^\mathsf{P}h(X)&=&\mathbb{E}^\mathsf{Q}\biggl(
\mathbh{1}_{\{
\tau^X>t_1\}}\mathbb{E}^\mathsf{Q}\biggl(\mathbh{1}_{\{\tau^X>t_2\}
}
\cdots\mathbb{E}^\mathsf{Q}\biggl(\mathbh{1}_{\{\tau^X>t_{n-2}\}
}\overline{g}
\biggl(\frac{1}{X}\biggr)\Big|\mathcal{F}_{t_{n-3}}\biggr)\cdots \Big|
\mathcal{F}_{t_1}\biggr)\biggr)
\\
&& {}- \mathbb{E}^\mathsf{Q}\bigl(\mathbh{1}_{\{t_{n-2}<\tau^X\leq
t_{n-1}\}}\eta
_{n-2}\bigl(X^{n-2}\bigr)\bigr) - \mathbb{E}^\mathsf{Q}
\bigl(\mathbh{1}_{\{
t_{n-1}<\tau^X\leq t_n\}}\eta_{n-1}\bigl(X^{n-1}\bigr)
\bigr).
\end{eqnarray*}
Iterating this procedure results in
\begin{eqnarray*}
\mathbb{E}^\mathsf{P}h(X)&=&\mathbb{E}^\mathsf{Q} \biggl(\mathbh
{1}_{\{\tau^X>t_1\}}\overline{g} \biggl(\frac
{1}{X} \biggr) \biggr)-\sum
_{k=1}^{n-1} \mathbb{E}^\mathsf{Q}
\bigl(\mathbh{1}_{\{
t_{k}<\tau^X\leq t_{k+1}\}}\eta_{k} \bigl(X^{k} \bigr)
\bigr)
\\
&=&\mathbb{E}^\mathsf{Q}\overline{g} \biggl(\frac{1}{X} \biggr)-
\mathbb{E}^\mathsf{Q} (\mathbh{1}_{\{\tau
^X\leq t_1\}}\eta_0 )-\sum
_{k=1}^{n-1} \mathbb{E}^\mathsf {Q}
\bigl(\mathbh{1}_{\{
t_{k}<\tau^X\leq t_{k+1}\}}\eta_{k} \bigl(X^{k} \bigr)
\bigr).
\end{eqnarray*}
\upqed\end{pf}

%
\begin{rem}
The sum following the minus sign in the above decompositions (\ref
{d1}) and (\ref{d2}) will be called the \textit{default term}. This
is motivated by the following observation:
%
\begin{eqnarray}
\gamma_X(t,T)&=& X_{t}-\mathbb{E}^\mathsf{P}(X_{T}|
\mathcal {F}_{t})=X_t-X_t\cdot\mathsf{Q}\bigl(
\tau ^X>T|\mathcal{F}_t\bigr)
\nonumber
\\[-8pt]
\label{bub}\\[-8pt]
\nonumber
&=& X_t\cdot
\mathsf{Q}\bigl(\tau^X\leq T| \mathcal {F}_t\bigr)\qquad
\mathsf{P}\mbox{-a.s.}
\end{eqnarray}
Here, the second equality in (\ref{bub}) is justified by the following
calculation, valid for any $\mathcal{F}_t$-measurable set $A$:
\begin{eqnarray*}
\mathbb{E}^\mathsf{P}(\mathbh{1}_AX_T)&=&
\mathsf{Q}\bigl(A,\tau ^X>T\bigr)=\mathsf{Q}\bigl(A,
\tau^X>t,\tau^X>T\bigr)\\
&=&\mathbb{E}^\mathsf{Q}
\bigl(\mathbh{1}_{\{A,\tau^X>t\}}\mathsf{Q}\bigl(\tau^X>T|\mathcal
{F}_t\bigr) \bigr)
\\
&=&\mathbb{E}^\mathsf{P} \bigl(\mathbh{1}_AX_t
\cdot\mathsf{Q}\bigl(\tau ^X>T|\mathcal{F}_t\bigr) \bigr)
\qquad \mathsf{P}\mbox{-a.s.}
\end{eqnarray*}
Taking expectations with respect to $\mathsf{P}$ in (\ref{bub}) yields
\begin{eqnarray*}
\mathbb{E}^\mathsf{P}\gamma_X(t,T)&=& \mathbb{E}^\mathsf{P}
\bigl(X_t\cdot\mathsf{Q}\bigl(\tau^X\leq T|\mathcal{F}
_t\bigr) \bigr)=\mathbb{E}^\mathsf{Q} \bigl(
\mathbh{1}_{\{\tau^X>t\}
}\mathsf{Q}\bigl(\tau^X\leq T|\mathcal{F}
_t\bigr) \bigr)\\
&=& \mathsf{Q}\bigl(t<\tau^X\leq T\bigr).
\end{eqnarray*}
Thus, the default term is directly related to the expected bubble of
the underlying. It measures how much the failure of the martingale
property by $X$ affects the option price. If $X$ is a true martingale,
it will equal zero.
\end{rem}

The convergence conditions that must be fulfilled in Theorem~\ref{inbetween} may seem to be rather strict. However, below we give a few
examples of options which satisfy those conditions.

%
\begin{ex}
Let us consider a modified call option with maturity $T$ and strike
$K$, where the holder has the option to reset the strike value to the
current stock price at certain points in time $t_1<t_2<\cdots<t_n<T$,
that is, the payoff profile of the option is given by
\[
H(X)=\bigl(X_T-\min(K,X_{t_1},X_{t_2},
\ldots,X_{t_n})\bigr)^+.
\]
With the notation in Theorem~\ref{inbetween} and setting $t_{n+1}=T$,
it follows that
\[
\eta_0=\eta_1=\cdots=\eta_{n}=1
\]
and the option value can be decomposed as
\begin{eqnarray*}
\mathbb{E}^\mathsf{P}h(X)&=&\mathbb{E}^\mathsf{Q} \biggl(1-
\frac
{1}{X_T}\cdot\min (K,X_{t_1},\ldots,X_{t_n} )
\biggr)^+-\sum_{k=0}^{n}\mathsf {Q}
\bigl(t_k<\tau^X\leq t_{k+1} \bigr)
\\
&=&\mathbb{E}^\mathsf{Q} \biggl(1-\frac{1}{X_T}\cdot\min
(K,X_{t_1},\ldots,X_{t_n} ) \biggr)^+-\gamma_X(0,T).
\end{eqnarray*}
Therefore, this modified call option has the same default as the normal
call option (cf. equation (14) in \cite{PalP}).
\end{ex}

%
\begin{ex}
Let us consider a call option on the ratio of the stock price at times
$T$ and $S\leq T$ with strike $K\in\mathbb{R}_+$, that is,
\[
h(X)= \biggl(\frac{X_T}{X_S}-K \biggr)^+
\]
for $S<T\in\mathbb{R}_+$. In this case,
\[
\eta_0=0, \qquad \eta_1(y)=y
\]
and the decomposition of the option value is given by
\[
\mathbb{E}^\mathsf{P}h(X)=\mathbb{E}^\mathsf{Q} \biggl(
\frac
{1}{X_S}-\frac{K}{X_T} \biggr)^+-\mathbb{E} ^\mathsf{Q}
\biggl(\mathbh{1}_{\{S<\tau^X\leq T\}}\frac
{1}{X_S} \biggr).
\]
\end{ex}

%
\begin{ex}
A \textit{chooser option} with maturity $T$ and strike $K$ entitles
the holder to decide at time $S<T$, whether the option is a call or a
put. He will choose the call, if its value is as least as high as the
value of the put option with strike $K$ and maturity $T$ at time $S$.
However, in the presence of asset price bubbles, that is, when the
underlying is a strict local martingale, put-call-parity does not hold,
but instead we have
\[
\mathbb{E}^\mathsf{P}\bigl((X_T-K)^+|\mathcal{F}_S
\bigr)-\mathbb{E}^\mathsf {P}\bigl((K-X_T)^+|
\mathcal{F}_S\bigr)=\mathbb{E}^\mathsf{P}(X_T|
\mathcal{F}_S)-K.
\]
Therefore, the payoff of the chooser option equals
\[
h(X_S,X_T)=(X_T-K)^+\mathbh{1}_{\{\mathbb{E}^\mathsf{P}(X_T|\mathcal
{F}_S)\geq K\}}+(K-X_T)^+
\mathbh{1}_{\{\mathbb{E}
^\mathsf{P}(X_T|\mathcal{F}_S)<K\}}.
\]
Let us assume that $X$ is Markovian. Then we can express $\mathbb
{E}^\mathsf{P}
(X_T|\mathcal{F}_S)$ as a function of $X_S$, say $\mathbb{E}^\mathsf
{P}(X_T|\mathcal{F}_S)=m(X_S)$,
and the limits defined in Theorem~\ref{inbetween} exist, if $m$ is
monotone for large values, and equal
\[
\eta_1(y)=\mathbh{1}_{ \{m ({1}/{y} )\geq
K \}}, \qquad \eta_0=\lim
_{x\rightarrow\infty}\mathbh{1}_{ \{m(x)\geq
K \}}.
\]
Thus, the value of the chooser option can be decomposed as
\begin{eqnarray*}
\mathbb{E}^\mathsf{P}h(X_S,X_T)&=&
\mathbb{E}^\mathsf{Q} \biggl(\frac{h(X_S,X_T)}{X_T} \biggr)-\mathsf {Q}
\bigl(m(X_S)\geq K, S<\tau^X\leq T \bigr)
\\
&&{}-\lim_{x\rightarrow\infty}\mathbh {1}_{ \{
m(x)\geq K \}}\mathsf{Q}\bigl(
\tau^X\leq S\bigr).
\end{eqnarray*}
If $X$ is the reciprocal of a BES(3)-process under $\mathsf{P}$, it is
calculated in Section~2.2.2 in \cite{CoxHobson} that
\[
m(X_S)=\mathbb{E}^\mathsf{P}(X_T|X_S)=X_S
\biggl(1-2\Phi \biggl(-\frac
{1}{X_S\sqrt
{T-S}} \biggr) \biggr).
\]
Therefore,
\begin{eqnarray*}
\lim_{x\rightarrow\infty}m(x) &=& \lim_{x\rightarrow\infty}\mathbb
{E}^\mathsf{P}(X_T|X_S=x)= \lim
_{x\rightarrow\infty}2\varphi \biggl(-\frac{1}{x\sqrt{T-S}} \biggr)
\frac
{1}{\sqrt{T-S}}\\
&=& \frac{\sqrt{2}}{\sqrt{\pi(T-S)}}
\end{eqnarray*}
and
\[
\eta_1(y)=\mathbh{1}_{ \{{1}/{y} (1-2\Phi
(-{y}/{\sqrt{T-S}} ) )\geq K \}}, \qquad\eta_0=\mathbh
{1}_{ \{
{\sqrt{2}}/{\sqrt{\pi(T-S)}}>K \}}.
\]
\end{ex}

%
\begin{rem}\label{price}
Here, we take the approach of valuating options by risk-neutral
expectations. While there may be other approaches, risk-neutral
expectations do not create arbitrage in the market, even though the
stock itself is not priced that way. Indeed, $\mathsf{P}$ remains an
ELMM in
the enlarged market also after adding any asset $V_t = \mathbb
{E}^\mathsf{P}[H | \mathcal{F}
_t], t\leq T$, for some integrable $H\in\mathcal{F}_T$.
Interestingly, by
choosing $H=X_T$ we may have $V_0< X_0$ (in the case where $X$ is a
strict local martingale). But it
is impossible to short $X$ and take a long position on $V$ all the way
up to $T$ because of credit constraints, therefore, NFLVR is not violated.
\end{rem}

In the following, we give another extension of Proposition~7 in \cite{PalP} to Barrier options, that is, we allow the options to be
knocked-in or knocked-out by passing some pre-specified level.

%
\begin{thm}\label{barrier}
Consider any nonnegative Borel-measurable function $h\dvtx \mathbb
{R}_{++}\rightarrow\mathbb{R}
_+$ and define $g(x)=x\cdot h (\frac{1}{x} )$ for $x>0$.
Suppose that $\lim_{x\rightarrow0} g(x)=:\eta<\infty$ exists and
denote by
$\overline{g}\dvtx\mathbb{R}_+\rightarrow\mathbb{R}_+$ the extension of
$g$ with $\overline
{g}(0)=\eta$. Define $\hat{m}^X_T:=\min_{t\leq T}X_t$, $m_T^X:=\max_{t\leq T}X_t$ as well as $T^X_a:=\inf\{t\geq0\dvtx X_t\leq a\}$ for
$a\in\mathbb{R}_+$. Then for any bounded stopping time $T$ and for
any real
numbers $D\leq1$ and $F\geq1$:
\renewcommand{\theequation}{DI}
\begin{equation}
\label{di}\hspace*{17pt}\mathbb{E}^\mathsf{P} \bigl(h(X_T)\mathbh{1}_{\{\hat
{m}_T^X\leq D\}}
\bigr)=\mathbb{E}^\mathsf{Q} \biggl(g \biggl(\frac{1}{X_T} \biggr)
\mathbh{1}_{\{\hat{m}_T^X\leq D\}
} \biggr)-\eta\cdot\mathsf{Q} \bigl(T^X_{D}<
\tau^X\leq T \bigr),
\end{equation}\vspace*{-10pt}
\renewcommand{\theequation}{DO}
\begin{equation}\label{do}
\quad\qquad\hspace*{-7pt}\mathbb{E}^\mathsf{P} \bigl(h(X_T)\mathbh{1}_{\{\hat
{m}_T^X\geq D\}}
\bigr)=\mathbb{E}^\mathsf{Q} \biggl(g \biggl(\frac{1}{X_T} \biggr)
\mathbh{1}_{\{\hat{m}_T^X\geq D\}
} \biggr)-\eta\cdot\mathsf{Q} \bigl(T^X_{D}=
\infty,\tau^X\leq T \bigr),\hspace*{-8pt}
\end{equation}\vspace*{-10pt}
\renewcommand{\theequation}{UI}
\begin{equation}
\label{ui}\hspace*{-12pt}\mathbb{E}^\mathsf{P} \bigl(h(X_T)\mathbh{1}_{\{m_T^X\geq F\}
}
\bigr)=\mathbb{E}^\mathsf{Q} \biggl(g \biggl(\frac{1}{X_T} \biggr)
\mathbh{1}_{\{m_T^X\geq F\}} \biggr)-\eta\cdot \mathsf{Q} \bigl(\tau^X
\leq T \bigr),
\end{equation}\vspace*{-10pt}
\renewcommand{\theequation}{UO}
\begin{equation}\label{uo}
\hspace*{-89pt}\mathbb{E}^\mathsf{P} \bigl(h(X_T)\mathbh{1}_{\{m_T^X\leq F\}
}
\bigr)=\mathbb{E}^\mathsf{Q} \biggl(g \biggl(\frac{1}{X_T} \biggr)
\mathbh{1}_{\{m_T^X\leq F\}} \biggr).
\end{equation}
\end{thm}

Before proving the theorem, we remark that the result is intuitively
reasonable because the default only plays a role if the option is
active. Especially note that the default term for Up-and-Out options
(\ref{uo}) is equal to zero, since in this case we can replace $X$ by the
uniformly\vspace*{1pt} integrable martingale $X^{\tilde{\tau}^X_{F+1}}$ in the
definition of the option's payoff function, where $\tilde{\tau
}^X_a:=\inf\{t\geq0\dvtx X_t>a\}$ for any $a\geq1$.

\begin{pf*}{Proof of Theorem~\ref{barrier}}
Keeping in mind that $D\leq1$ and $F\geq1$, it follows from the
absolute continuity relationship between $\mathsf{P}$ and $\mathsf
{Q}$ that
\begin{eqnarray*}
\mathbb{E}^\mathsf{P} \bigl(h(X_T)\mathbh{1}_{\{\hat{m}_T^X\leq D\}
}
\bigr)&=& \mathbb{E}^\mathsf{Q} \biggl(g \biggl(\frac{1}{X_T} \biggr)
\mathbh {1}_{\{\tau^X>T, \hat
{m}_T^X\leq D\}} \biggr)
\\
&=&  \mathbb{E}^\mathsf{Q} \biggl(g \biggl(
\frac{1}{X_T} \biggr)\mathbh {1}_{\{\tau^X>T\geq T^X_D\}
} \biggr)
\\
&=&\mathbb{E}^\mathsf{Q} \biggl(g \biggl(\frac{1}{X_T} \biggr)
\mathbh {1}_{\{\hat{m}_T^X\leq
D\}} \biggr)-\eta\cdot\mathsf{Q} \bigl(T^X_{D}
\leq T, \tau^X\leq T \bigr)
\\
&=&\mathbb{E}^\mathsf{Q} \biggl(g \biggl(\frac{1}{X_T} \biggr)
\mathbh {1}_{\{\hat{m}_T^X\leq
D\}} \biggr)-\eta\cdot\mathsf{Q} \bigl(T^X_{D}<
\tau^X\leq T \bigr).
\end{eqnarray*}
This proves the formula for the Down-and-In barrier option (\ref{di}). The
other three formulas can be proven in a similar way by noting that
\begin{eqnarray*}
\mathsf{Q} \bigl(\tau^X\leq T<T^X_D
\bigr)&=&\mathsf{Q} \bigl(\tau ^X\leq T, T^X_D=
\infty \bigr),
\\
\mathsf{Q} \bigl(\tilde{\tau}^X_{F}\leq T,
\tau^X\leq T \bigr)&=&\mathsf{Q} \bigl(\tau^X\leq T
\bigr),
\\
\mathsf{Q} \bigl(\tau^X\leq T<\tilde{\tau}^X_{F}
\bigr)&=&0.
\end{eqnarray*}
\upqed\end{pf*}

%
\begin{rem}
Above we used the risk-neutral pricing approach to calculate the value
of some options written on a stock which may have an asset price
bubble, as suggested by the first fundamental theorem of asset pricing.
The derived decompositions show that there is an important difference
in the option value depending on whether the underlying is a strict
local or a true martingale under the risk-neutral measure, which is
reflected in the default term. Even though we do not create arbitrage
opportunities when pricing options by their fundamental values
calculated above, several authors have suggested to ``correct'' the
option price to account for the strictness of the local martingale
(cf., e.g., \cite{Carr,JPffbubbles,JPScomp,JPSincomp,MadanYor}). In~\cite{Carr}, the price of a contingent claim is defined as the minimal
super-replicating cost under both measures $\mathsf{P}$ and $\mathsf{Q}$
corresponding to two different currencies, where the process $X$ is
interpreted as the exchange rate between them. While the authors of
\cite{JPffbubbles,JPScomp,JPSincomp}
work under the additional No Dominance assumption, which is strictly
stronger than NFLVR, and allow for bubbles in the option prices within
this framework, in \cite{MadanYor} the following pricing formulas for
European and American call options written on (continuous) $X$ with
strike $K$ and maturity $T$ are suggested:
\begin{eqnarray*}
C^{\mathrm{strict}}_E(K,T)&:=&\lim_{n\rightarrow\infty}
\mathbb{E}^\mathsf {P}(X_{T\wedge\sigma
_n}-K)^+,
\\
C^{\mathrm{strict}}_A(K,T)&:=&\sup_{\sigma\in\mathcal{T}_{0,T}}\lim
_{n\rightarrow
\infty} \mathbb{E}^\mathsf{P}(X_{\sigma\wedge\sigma_n}-K)^+
\end{eqnarray*}
for some localizing sequence $(\sigma_n)_{n\in\mathbb{N}}$ of the (strict)
local martingale $X$. It is proven in \cite{MadanYor} that these
definitions are independent of the chosen localizing sequence and that
$C_E^{\mathrm{strict}}=C_A^{\mathrm{strict}}$. However, a generalization of this
definition to any other option $h(\cdot)$ on $X$ with maturity $T$ is
problematic: the independence of the chosen localizing sequence
$(\sigma_n)_{n\in\mathbb{N}}$ is not true in general, so one may
have to
choose $\sigma_n=\tau_n^X$ as defined above. Moreover, in general
$\lim_{n\rightarrow\infty}\mathbb{E}^\mathsf{P}h(X_T^{\sigma_n})$
may not be well defined
and equal to $\mathbb{E}^\mathsf{P}h(X_T)$, even when $X$ is a true
martingale, as
the following example shows.
\end{rem}

%
\begin{ex}
Suppose that $(\log(X_t)+t/2)_{t\geq0}$ is a Brownian motion, that
is, $X$ is a geometric Brownian motion, and consider the claim $h(X_T)$
with continuous payoff function
\[
h(x)=\sum_{n \in\mathbb{N}} \mathbh{1}_{\{n-a_n\leq x\leq n+a_n\}
}f_n
\biggl(n-\frac
{n|x-n|}{a_n} \biggr) \qquad \mbox{with } f_n(z)=
\frac{1}{\mathsf{P}(\tau
_n^X\leq
1)}\cdot\frac{z}{n},
\]
where each $a_n\in(0,1)$ is chosen small enough such that
\[
2n^2\cdot\mathsf{P}(n-a_n\leq X_1\leq
n+a_n)\leq\mathsf{P}\bigl(\tau ^X_n\leq1
\bigr).
\]
Let us set $T=1$ and $\sigma_n=\tau_n^X$ for all $n\in\mathbb{N}$.
In this case,
\[
\mathbb{E}^\mathsf{P}h(X_{1\wedge\tau^X_n})\geq\mathsf{P}\bigl(\tau
_n^X\leq1\bigr)f_n(n)=1, \qquad n\in \mathbb{N},
\]
but
\begin{eqnarray*}
\mathbb{E}^\mathsf{P}h(X_1)&\leq& \sum
_{n\in\mathbb{N}}\mathsf {P} (n-a_n\leq X_1\leq
n+a_n )f_n(n)
\\
&\leq&\sum_{n\in\mathbb{N}}\frac{\mathsf{P}(\tau^X_n\leq
1)}{2n^2}\cdot
f_n(n)=\frac{\pi^2}{12}<1.
\end{eqnarray*}

Since in this example there are no asset price bubbles, it does not
seem correct to trade the option for a price which differs from its
fundamental value. Therefore, in the case where we have a decomposition
of the fundamental option value as above or more generally as proven in
Theorem~\ref{inbetween}, this suggests that the most sensible approach
to correct the option value for bubbles in the underlying is to set the
default term equal to zero. Equivalently, we can also set $\tau^X$
equal to infinity under the measure $\mathsf{Q}$. This even gives a
way of
correcting the option value for stock price bubbles in the general
case, where a decomposition formula may not be available, leaving open
the question of why this should give an arbitrage-free pricing rule. By
doing so, we would basically treat the price process as if it were a
true martingale. However, we want to emphasize that it is not necessary
to correct the price at all, since the fundamental value gives an
arbitrage-free price as explained in Remark~\ref{price}.
\end{ex}

\section{Relationship between $\mathsf{P}$ and $\mathsf{Q}$}\label{PQ}

In the following, we study the relationship between the original
measure $\mathsf{P}$ and the measure $\mathsf{Q}$ in more detail. We
suppose that
assumption (S) is valid throughout the entire section.

%
\begin{lem}
Set $X=\tilde{X}$, that is, $X_t=\infty$ on $\{t\geq\tau^X\}$.
Then, $\mathsf{Q}(X_\infty=\infty)=1  \Leftrightarrow  \mathsf
{P}(X_\infty=0)=1$.
\end{lem}

\begin{pf}
Since $X$ is a $\mathsf{P}$-super-martingale and $\frac{1}{X}$ a
$\mathsf{Q}$-martingale, both converge and, therefore, $X_\infty$ is almost
surely well defined under both measures.

\underline{$\Leftarrow$}: Assume that $\mathsf{P}(X_\infty=0)=1$.
Because $1/X$ is a $\mathsf{Q}
$-martingale, we have by Fatou's lemma for all $u>0$,
\begin{eqnarray*}
\mathbb{E}^\mathsf{Q} \biggl(\frac{1}{X_\infty}\mathbh{1}_{\{\tau
^X>t,X_t>u\}}
\biggr)&\leq & \liminf_{n\rightarrow\infty}\mathbb{E}^\mathsf{Q}
\biggl(\frac
{1}{X_{t+n}}\mathbh{1}_{\{\tau
^X>t,X_t>u\}} \biggr)\\
&=& \mathbb{E}^\mathsf{Q}
\biggl(\frac
{1}{X_t}\mathbh{1}_{\{\tau
^X>t,X_t>u\}} \biggr)
\\
&=&\mathsf{P}(X_t>u).
\end{eqnarray*}
By dominated convergence for $t\rightarrow\infty$,
\[
\mathbb{E}^\mathsf{Q} \biggl(\frac{1}{X_\infty}\mathbh{1}_{\{\tau
^X=\infty,X_\infty>u\}
}
\biggr)\leq\mathsf{P}(X_\infty\geq u)=0 \qquad \forall u>0.
\]
This implies that
\[
\mathbb{E}^\mathsf{Q} \biggl(\frac{1}{X_\infty}\mathbh{1}_{\{\tau
^X=\infty,X_\infty>0\}
}
\biggr)=0.
\]
Since $\frac{1}{X}$ is a $\mathsf{Q}$-martingale,
\[
\mathbb{E}^\mathsf{Q} \biggl(\frac{1}{X_\infty} \biggr)\leq\mathbb
{E}^\mathsf{Q} \biggl(\frac
{1}{X_t} \biggr)=1.
\]
Thus, $\mathsf{Q}(X_\infty=0)=0$ and
\[
\mathbb{E}^\mathsf{Q} \biggl(\frac{1}{X_\infty}\mathbh{1}_{\{\tau
^X=\infty\}}
\biggr)=0  \quad \Leftrightarrow\quad  \frac{1}{X_\infty}\mathbh{1}_{\{\tau^X=\infty\}
}=0\qquad
\mathsf{Q}\mbox{-a.s.}
\]
Since $\frac{1}{X_\infty}\mathbh{1}_{\{\tau^X<\infty\}}=0$, it
follows that
$\frac{1}{X_\infty}=0$ $\mathsf{Q}$-almost surely.

\underline{$\Rightarrow$}: Assume that $\mathsf{Q}(X_\infty=\infty
)=1$. Because $X$ is a
$\mathsf{P}$-super-martingale, we have
\[
\mathbb{E}^\mathsf{P}X_\infty\leq\mathbb{E}^\mathsf{P}X_t
\leq1
\]
and
\begin{eqnarray}
\mathbb{E}^\mathsf{P} (X_\infty\mathbh{1}_{\{X_t< k\}} )\leq
\mathbb{E}^\mathsf{P} (X_t\mathbh{1} _{\{X_t< k\}} )=
\mathsf{Q}\bigl(t<\tau^X,X_t< k\bigr)=
\mathsf{Q}(X_t< k)\nonumber
\\
\eqntext{\forall k\geq0.}
\end{eqnarray}
For $t\rightarrow\infty$ by dominated convergence then
\[
\mathbb{E}^\mathsf{P} (X_\infty\mathbh{1}_{\{X_\infty< k\}
} )\leq
\mathsf{Q}(X_\infty< k)=0\qquad  \forall k\geq0.
\]
This implies that $X_\infty\mathbh{1}_{\{X_\infty< k\}}=0$
$\mathsf{P}$-a.s. for all
$k\geq0$. Therefore, $\mathsf{P}(X_\infty\in\{0, \infty\})=1$.
Since $\mathbb{E}^\mathsf{P}(X_\infty)\leq1$, it follows that $\mathsf{P}(X_\infty
=\infty)=0$, and
thus $X_\infty=0$ $\mathsf{P}$-almost surely.
\end{pf}

Until here, we have only considered the behaviour of the local $\mathsf{P}$-martingale $X$ under
$\mathsf{Q}$. But how do other processes change their
behaviour, when passing from $\mathsf{P}$ to $\mathsf{Q}$? This
question is of
particular interest, since we want to apply our results to the pricing
of options written on more than one underlying stock. Let us assume
that besides $X$ there exists another process $Y$ on $(\Omega,\mathcal
{F},(\mathcal{F}
_t)_{t\geq0},\mathsf{P})$. For all $n\in\mathbb{N}$ we set $\tau
_n^Y=\inf\{t\geq
0\dvtx Y_t>n\}\wedge n$ and $\tau^Y=\lim_{n\rightarrow\infty}\tau
_n^Y$. Note
that in what follows we identify $Y$ with the process $\tilde{Y}$
defined above.

%
\begin{lem}\label{explosion} Let $Y$ be a nonnegative c\`adl\`ag
local $\mathsf{P}$-martingale. Then $\mathsf{Q}( \tau^X\leq
\tau^Y ) = 1$.
\end{lem}

\begin{pf}
\[
\mathsf{Q}\bigl(\tau^Y<\tau^X\bigr)=\lim
_{n\rightarrow\infty}\mathsf {Q}\bigl(\tau^Y<
\tau^X_n\bigr)=\lim_{n\rightarrow\infty}
\mathbb{E}^\mathsf{P} (X_{\tau
_n^X}\mathbh{1}_{\{\tau^Y<\tau_n^X\}
} )=0.
\]
\upqed\end{pf}

Moreover, we introduce condition (T): $\mathsf{Q}(\tau
^X=\tau^Y<\infty)=0$.

Clearly, (T) is always fulfilled if $X$ is a true martingale. Moreover,
condition~(T) also holds, if $X$ and $Y$ are independent under $\mathsf{P}$.
Indeed, in this case for every $n\in\mathbb{N}$
\begin{eqnarray*}
&& \mathsf{Q}\bigl(\tau^Y=\tau^X<n\bigr)\\
&& \qquad =\lim
_{m\rightarrow\infty}\mathsf {Q}\bigl(\tau^Y_m<\tau
^X<n\bigr)=\lim_{m\rightarrow\infty}\lim_{k\rightarrow\infty}
\mathsf {Q}\bigl(\tau^Y_m<\tau ^X_k<n
\bigr)
\\
&&\qquad=\lim_{m\rightarrow\infty}\lim_{k\rightarrow\infty}\mathbb
{E}^\mathsf{P} (X_{\tau^X_k}\mathbh{1} _{\{\tau^Y_m<\tau^X_k<n\}} )\leq\lim
_{m\rightarrow\infty
}\lim_{k\rightarrow\infty}\mathbb{E}^\mathsf{P}
(X_{\tau
^X_k}\mathbh{1}_{\{\tau^Y_m<n\}} )
\\
&&\qquad= \lim_{m\rightarrow\infty}\lim_{k\rightarrow\infty}\mathbb
{E}^\mathsf{P}X_{\tau
^X_k}\cdot\mathsf{P}\bigl(\tau^Y_m<n
\bigr)=\lim_{m\rightarrow\infty}\mathsf {P}\bigl(\tau^Y_m<n
\bigr)=0.
\end{eqnarray*}

However, in general it is hard to check condition (T), since it
requires some knowledge of the joint distribution of $\tau^X_n$ and
$\tau^Y_m$ for $n,m$ large.

If $X$ and $Y$ are assumed to be c\`adl\`ag processes under $\mathsf
{P}$, they
are also almost surely c\`adl\`ag under $\mathsf{Q}$ before time $\tau^X$
because $\mathsf{P}$ and $\mathsf{Q}$ are equivalent on every
$\mathcal{F}_{\tau^X_n}$.
Furthermore, since $\frac{1}{X}$ is a $\mathsf{Q}$-martingale, it
does not
explode and, therefore, $X_{t-}\neq0$ and $X_t\neq0$ $\mathsf{Q}$-almost
surely for all $t\geq0$. Thus, the process $Z:=\frac{Y}{X}$ does also
have almost surely c\`adl\`ag paths before time $\tau^X$. Since from
time $\tau^X$ on everything is constant, the only crucial question is
whether $Z=\frac{Y}{X}$ has a left-limit at $\tau^X$.

%
\begin{lem}\label{yyy}
Let $Y$ be a nonnegative local $\mathsf{P}$-martingale. Then
$Z_t:=
(\frac{Y_{t}}{X_{t}} )_{0\leq t<\tau^X}$ is a local martingale
on $(\Omega,\mathcal{F}_{\tau^X-},(\mathcal{F}_{t\wedge\tau
^X-})_{t\geq0},\mathsf{Q})$.
Furthermore, setting $Z_t:=\tilde{Z}_t$ and $X_t=\infty$ on $\{t\geq
\tau^X\}$ is the unique way to define $Z$ and $X$ after time $\tau^X$
such that $\frac{1}{X}$ and $Z$ remain nonnegative c\`adl\`ag local
martingales on $[0,\infty)$ for all possible extensions of the measure
$\mathsf{Q}$ from $\mathcal{F}_{\tau^X-}$ to $\mathcal{F}=\bigvee_{t\geq0}\mathcal{F}_t$.
\end{lem}

\begin{pf}
First, we show that $Z=\frac{Y}{X}$ is a local $\mathsf{Q}$-martingale on
$\bigcup_{n\in\mathbb{N}}[0,\tau^X_n]$ with localizing sequence
$(\tau
_n^Y\wedge\tau_n^X)_{n\in\mathbb{N}}$. Indeed, we have for all
$t\geq0$
and $n\in\mathbb{N}$,
\begin{eqnarray*}
\mathbb{E}^\mathsf{Q} (Z_{\tau_n^Y\wedge\tau^X_n}|\mathcal {F}_t )&=&
\mathbb{E}^\mathsf{Q} \biggl(\frac{Y_{\tau_n^Y\wedge\tau
^X_n}}{X_{\tau
_n^Y\wedge\tau_n^X}}\Big|\mathcal{F}_t
\biggr)= \mathbb{E}^\mathsf{P} \biggl(\frac{Y_{\tau_n^Y\wedge\tau
^X_n}}{X_{t\wedge
\tau_n^Y\wedge\tau^X_n}}\Big|
\mathcal{F}_t \biggr)= \frac{Y_{t\wedge
\tau_n^Y\wedge\tau^X_n}}{X_{t\wedge\tau_n^Y\wedge\tau^X_n}}
\\
&=&Z_{t\wedge\tau_n^Y\wedge\tau^X_n}
\end{eqnarray*}
and by Lemma~\ref{explosion} we know that $\tau_n^X\wedge\tau
_n^Y\rightarrow\tau^X$ $\mathsf{Q}$-almost surely. Since $Z$ is a
nonnegative local
super-martingale up to time $\tau^X$, we can apply Fatou's lemma twice
with $s\leq t$:
\begin{eqnarray*}
\tilde{Z}_s&=&\liminf_{u\rightarrow\tau^X,u<\tau^X,u\in\mathbb
{Q}}Z_{s\wedge u}=
\liminf_{u\rightarrow\tau^X,u<\tau^X,u\in\mathbb{Q}}\lim_{n\rightarrow\infty
}Z_{s\wedge u\wedge\tau_n^X\wedge\tau_n^Y}
\\
&\geq&\liminf_{u\rightarrow\tau^X,u<\tau^X,u\in\mathbb{Q}}\lim_{n\rightarrow\infty}\mathbb{E}
^\mathsf{Q} ( Z_{t\wedge u\wedge\tau_n^X\wedge\tau
_n^Y}|\mathcal{F}_s ) \geq\liminf
_{u\rightarrow\tau^X,u<\tau^X,u\in\mathbb{Q}}\mathbb {E}^\mathsf{Q} (Z_{t\wedge
u}|
\mathcal{F}_s )
\\
&\geq& \mathbb{E}^\mathsf{Q} \Bigl( \liminf_{u\rightarrow\tau
^X,u<\tau^X,u\in\mathbb{Q}
}Z_{t\wedge u}|
\mathcal{F}_s \Bigr)=\mathbb{E}^\mathsf{Q}(\tilde
{Z}_t|\mathcal{F}_s),
\end{eqnarray*}
where the second inequality is due to the fact that $\mathbb
{E}^\mathsf{Q}
( Z_{t\wedge u\wedge\tau_n^X\wedge\tau_n^Y}|\mathcal{F}
_s )
\geq\mathbb{E}^\mathsf{Q} (Z_{t\wedge u}|\mathcal{F}_s
)$ by the super-martingale
property. By the convergence theorem for positive super-martingales, we
conclude that $\tilde{Z}_{\tau^X-}=Z_{\tau^X-}$ exists $\mathsf{Q}$-almost
surely in $\mathbb{R}_+$. To see that $\tilde{Z}$ is indeed a local
martingale and not only a super-martingale, we show that $\tilde
{Z}^{\tau_n^Z}$ is a uniformly integrable martingale for all $n\in
\mathbb{N}
$, where $\tau_n^Z=\inf\{t\geq0| Z_t>n\}\wedge n$. Since $\tilde
{Z}$ is a nonnegative super-martingale, it is sufficient to prove that
the expectation of $\tilde{Z}^{\tau_n^Z}$ is constant:
\begin{eqnarray*}
\mathbb{E}^\mathsf{Q}\tilde{Z}_{\tau_n^Z}&=&\mathbb{E}^\mathsf
{Q} (\tilde{Z}_{\tau_n^Z}\mathbh{1} _{\{\tau_n^Z<\tau^X\}}+\tilde{Z}_{\tau_n^Z}
\mathbh{1}_{\{\tau
_n^Z\geq\tau
^X\}} )
\\
&=&\lim_{m\rightarrow\infty}\mathbb{E}^\mathsf{Q} (Z_{\tau
_n^Z}
\mathbh{1}_{\{\tau_n^Z<\tau
_m^X\wedge\tau_m^Y\}} ) +\mathbb{E}^\mathsf{Q} (\tilde
{Z}_{\tau^X-}\mathbh{1} _{\{\tau_n^Z\geq\tau^X\}} )
\\
&=& \lim_{m\rightarrow\infty}\mathbb{E}^\mathsf{Q} (Z_{\tau
_m^X\wedge\tau_m^Y}
\mathbh{1} _{\{\tau_n^Z<\tau_m^X\wedge\tau_m^Y\}} )+\mathbb{E}^\mathsf {Q} \Bigl(\lim
_{m\rightarrow\infty} Z_{\tau^X_m\wedge\tau_m^Y}\mathbh{1}_{\{
\tau_n^Z\geq\tau
^X\}} \Bigr)
\\
&=& \lim_{m\rightarrow\infty}\mathbb{E}^\mathsf{Q}Z_{\tau
_m^X\wedge\tau_m^Y}-
\lim_{m\rightarrow\infty}\mathbb{E}^\mathsf{Q} (Z_{\tau^X_m\wedge
\tau_m^Y}
\mathbh{1}_{\{\tau
^X>\tau_n^Z\geq\tau^X_m\wedge\tau_m^Y\}} )=Z_0.
\end{eqnarray*}
%
To prove the uniqueness of the extension of $Z$ for all possible
extensions of $\mathsf{Q}$ to $\mathcal{F}$, define for all $n\in
\mathbb{N}$, $\tau
^Z_n=\inf\{t\geq0\dvtx \overline{Z}_t>n\}$, where $\overline{Z}$ is an
arbitrary
c\`adl\`ag extension of $(Z_t)_{t<\tau^X}$. Then $(\tau^Z_n)_{n\in
\mathbb{N}
}$ is a localizing sequence for $\overline{Z}$ for all possible extensions
of $\mathsf{Q}$. Fix one of these extensions and call it $\mathsf
{Q}^0$. We have
\[
\mathbb{E}^{\mathsf{Q}^0}\bigl(\overline{Z}_t^{\tau_n^Z}|
\mathcal {F}_s\bigr)=\overline{Z}_s^{\tau_n^Z}\qquad
\forall  n\in\mathbb{N}.
\]
Now for fix $n\in\mathbb{N}$ define the new measure $\mathsf{Q}^n$
on $\mathcal{F}$ via
\[
\frac{d\mathsf{Q}^n}{d\mathsf{Q}^0}=\frac{\overline{Z}_{\tau
^Z_n}}{\overline{Z}^{\tau
_n^Z}_{\tau^X-}}.
\]
Note that $\mathsf{Q}^n$ is also an extension of $\mathsf{Q}$ from
$\mathcal{F}_{\tau^X-}$
to $\mathcal{F}$. Furthermore, for all $\varepsilon\geq0$,
\begin{eqnarray*}
\overline{Z}_{\tau^X-}^{\tau_n^Z}&=&  \mathbb{E}^{\mathsf{Q}^n} \bigl(
\overline{Z}_{\tau^X+\varepsilon
}^{\tau_n^Z}|\mathcal{F} _{\tau^X-} \bigr)=
\mathbb{E}^{\mathsf{Q}^0} \biggl( \frac{\overline{Z}_{\tau
^Z_n}}{\overline{Z}^{\tau
_n^Z}_{\tau^X-}}\cdot\overline{Z}_{\tau^X+\varepsilon}^{\tau
_n^Z}\Big|
\mathcal{F} _{\tau^X-} \biggr)\\
&=&  \mathbb{E}^{\mathsf{Q}^0} \biggl(
\frac{ (\overline{Z}^{\tau
^Z_n}_{\tau
^X+\varepsilon} )^2}{\overline{Z}^{\tau_n^Z}_{\tau
^X-}}\Big|\mathcal{F}_{\tau
^X-} \biggr),
\end{eqnarray*}
because $\overline{Z}^{\tau_n^Z}$ must also be a uniformly integrable
martingale under $\mathsf{Q}^n$. Therefore, $\overline{Z}^{\tau
_n^Z}$ and $(\overline
{Z}^{\tau_n^Z})^2$ are both $\mathsf{Q}^0$-martingales\vspace*{1pt} after time
$\tau^X-$,
which implies that $\overline{Z}_{\varepsilon+\tau^X}=\overline
{Z}_{\tau^X-}$ for all
$\varepsilon\geq0$. Thus, $\overline{Z}\equiv\tilde{Z}$ is
uniquely determined.
\end{pf}

As usual to simplify notation, we will identify $Z$ with the process
$\tilde{Z}$ in the following.

%
\begin{rem}\mbox{}
\begin{itemize}
\item Note that if condition (T) is satisfied, then $Z_{\tau
^X}=Z_{\tau^X-}=0$ on $\{\tau^X<\infty\}$ $\mathsf{Q}$-almost surely.
\item Even though we proved that $Z_{\tau^X-}$ exists $\mathsf{Q}$-a.s. and
also $X_{\tau^X-}$ is well defined, this does not allow us to infer
any conclusions about the set $\{Y_{\tau^X-}\mbox{ exists in }\mathbb
{R}_+\}
$ in general.
\item For our purposes it is sufficient that local $\mathsf{Q}$-martingales
are c\`adl\`ag almost everywhere, since we are only interested in
pricing and do not deal with an uncountable number of processes. One
should, however, have in mind that in order to have \textit
{everywhere} regular paths some kind of augmentation is needed
(cf. \cite{KN}).
\end{itemize}
\end{rem}

%
\begin{rem}\label{ztx}
If $\Omega=C'(\mathbb{R}_+,\overline{\mathbb{R}}_+^2)$ is the path
space introduced in
Lem\-ma~\ref{canonical}, $(X,Y)$ is the coordinate process, and $(\tilde
{\mathcal{F}}_t)_{t\geq0}$ is the canonical filtration generated by $(X,Y)$,
then under the assumptions of Lemma~\ref{yyy} we can extend $\mathsf
{Q}$ to
$\mathcal{F}=\bigvee_{t\geq0}\mathcal{F}_t$ such that
\[
\mathsf{Q} \bigl(\omega_1(t)=\infty, \omega_2(t)=
\omega_2\bigl(\tau ^X-\bigr) \  \forall t\geq
\tau^X \bigr)=1.
\]
\end{rem}

%
\begin{lem}\label{XY}
Let $Y$ be a nonnegative local $\mathsf{P}$-martingale and set
$Z:=\frac{Y}{X}$.
\begin{enumerate}[(2)]
\item[(1)] If $X$ is a $\mathsf{P}$-martingale, then $Z$ is a strict local
$\mathsf{Q}
$-martingale if and only if $Y$ is a strict local $\mathsf{P}$-martingale.
\item[(2)] Assume that $X$ is a strict local $\mathsf{P}$-martingale. Then:
\begin{longlist}[(a)]
\item[(a)] If $Y$ is a $\mathsf{P}$-martingale, then $Z$ is a $\mathsf
{Q}$-martingale and
$Z_{\tau^X}=0$ on $\{\tau^X<\infty\}$.
\item[(b)] If $Z$ is a strict local $\mathsf{Q}$-martingale or $Z$ is a
$\mathsf{Q}
$-martingale with $\mathsf{Q}(\tau^X<\infty, Z_{\tau^X}>0)>0$,
then $Y$
is a strict local $\mathsf{P}$-martingale.
\item[(c)] If $Z$ is a $\mathsf{Q}$-martingale and if condition ($\mathrm{T}$) holds,
then $Y$
is a $\mathsf{P}$-martingale.
\item[(d)] If $Y$ is a strict local $\mathsf{P}$-martingale and if
condition ($\mathrm{T}$)
holds, then $Z$ is a~strict local $\mathsf{Q}$-martingale.
\end{longlist}
\end{enumerate}
\end{lem}

\begin{pf}
\begin{enumerate}[(2)]
\item[(1)] This is obvious, because $\mathsf{Q}$ and $\mathsf{P}$ are
locally equivalent,
if $X$ is a true $\mathsf{P}$-martingale.
\item[(2)] First note that
\begin{eqnarray*}
\mathbb{E}^\mathsf{P}Y_0 &=& \mathbb{E}^\mathsf{Q}Z_0
\geq \mathbb {E}^\mathsf{Q}Z_t=\mathbb{E}^\mathsf{Q}
(Z_t\mathbh{1}_{\{
t<\tau
^X\}} )+\mathbb{E}^\mathsf{Q}
(Z_t\mathbh{1}_{\{t\geq
\tau^X\}} )
\\
&=& \mathbb{E}^\mathsf{Q} \biggl(\frac{Y_t}{X_t}\mathbh{1}_{\{t<\tau^X\}
}
\biggr)+\mathbb{E}^\mathsf{Q} (Z_{\tau^X}\mathbh{1}_{\{t\geq\tau^X\}}
)
\\
&=&\mathbb{E}^\mathsf{P} {Y_t}+\mathbb{E}^\mathsf{Q}
(Z_{\tau
^X}\mathbh{1}_{\{t\geq\tau^X\}} )\geq\mathbb{E}^\mathsf{P}Y_t.
\end{eqnarray*}
\begin{longlist}[(a)]
\item[(a)] Since $Y$ is a positive local $\mathsf{P}$-martingale, we have
\begin{eqnarray*}
&& Y\mbox{ is a true $\mathsf{P}$-martingale }
\\
&&\qquad  \Leftrightarrow \quad \mathbb{E}^\mathsf{P}Y_t=
\mathbb{E}^\mathsf{P} Y_0\qquad \mbox{for all }t\geq0,
\\
&& \qquad\Leftrightarrow \quad \mathbb{E}^\mathsf{Q}Z_t =
\mathbb{E}^\mathsf {Q}Z_0 \qquad \mbox{for all }t\geq0,
Z_{\tau^X}\mathbh{1} _{\{\tau^X<\infty\}}=0 \ \mathsf{Q}\mbox{-a.s.}
\end{eqnarray*}
\item[(b)] Follows from (a).
%
\item[(c)]
If (T) holds, $Z_{\tau^X}=0$ on $\{\tau^X<\infty\}$ $\mathsf{Q}$-almost
surely; cf. Remark~\ref{ztx}. Therefore, since $Z$ is a $\mathsf{Q}
$-martingale, the above inequality turns into an equality and $Y$ is a
true $\mathsf{P}$-martingale.
\item[(d)] Follows from (c).
\end{longlist}
\end{enumerate}
\upqed\end{pf}

%
\begin{ex}[(Continuation of Example~\ref{difex})]
For the following example, we work on the path space $C'(\mathbb
{R}_+,\overline{\mathbb{R}}_+^2)$ with $(X,Y)$ denoting the coordinate process and $(\mathcal{F}
_t)_{t\geq0}$ being the right-continuous augmentation of the canonical
filtration generated by the coordinate process. Remember from Example~\ref{difex} that for $\sigma(x)$ locally bounded and bounded away
from zero for $x>0$, $\sigma(0)=0$, the local $\mathsf{P}$-martingale
\[
dX_t=\sigma(X_t)\,dW_t, \qquad X_0=1,
\]
is strictly positive whenever
\[
\int_0^1\frac{x}{\sigma^2(x)}\,dx=\infty,
\]
and under $\mathsf{Q}$ with $ \frac{d\mathsf{P}}{d\mathsf
{Q}}|_{\mathcal{F}_t}=\frac
{1}{X_t}$ the reciprocal process is a true martingale with decomposition
\[
d \biggl(\frac{1}{X_t} \biggr)=-\frac{\sigma(X_t)}{X_t^2}\,dW_t^\mathsf
{Q}=\overline {\sigma} \biggl(\frac{1}{X_t} \biggr)\,dW_t^\mathsf{Q}
\]
for the $\mathsf{Q}$-Brownian motion $W^{\mathsf{Q}}_t=W_t-\int_0^t\frac{\sigma
(X_s)}{X_s}\,ds$ defined on the set $\{t<\tau^X\}$ and $\overline{\sigma
}(y):=-y^2\cdot\sigma (\frac{1}{y} )$.

Now let us assume that $Y$ is also a local martingale under $\mathsf
{P}$ with dynamics
\[
dY_t=\gamma(Y_t)\,dB_t,
\]
where $\gamma$ fulfills the same assumptions as $\sigma$ and $B$ is
another $\mathsf{P}$-Brownian motion such that \mbox{$\langle
B,W\rangle
_t=\rho t$.} Then $\frac{Y}{X}$ is a $\mathsf{Q}$-local martingale
with decomposition
\[
d \biggl(\frac{Y_t}{X_t} \biggr)=\frac{\gamma(Y_t)}{X_t}\,dB_t^\mathsf{Q}
+Y_t\overline{\sigma} \biggl(\frac{1}{X_t} \biggr)\,dW_t^\mathsf{Q},
\]
where $B^\mathsf{Q}$ is a $\mathsf{Q}$-BM defined up to time $\tau
^X$ such that
$\langle B^\mathsf{Q},W^\mathsf{Q}\rangle_t=\rho t$ on $\{t<\tau^X\}$.
\end{ex}

\section{Application to financial bubbles II: Last passage time
formulas}\label{app2}


In Section~\ref{app1}, we have seen how one can determine the
influence bubbles have on option pricing formulas through a
decomposition of the option value into a ``normal'' term and a default
term (cf. Theorems \ref{inbetween} and \ref{barrier}). However, this
approach only works well for options written on one underlying. It is
rather difficult to give a universal way of how to determine the
influence of asset price bubbles on the valuation of more complicated
options and we will not do this here in all generality. Instead, we
will do the analysis for a special example, the so called exchange
option, which allows us to connect results about last passage times
with the change of measure that was defined in Section~\ref{newmeasure}.

Again we suppose that assumption (S) holds throughout the entire
section. In addition, we assume that there exists another strictly
positive process $Y$ on $(\Omega,\mathcal{F},(\mathcal{F}_t)_{t\geq
0},\mathsf{P})$, which is
also a local $\mathsf{P}$-martingale. Furthermore, in the following we will
assume that\vspace*{1pt} $X$ and $Y$ are \textit{continuous}. As in Section~\ref{PQ},
we define $Z:=\frac{Y}{X}$, which is a local $\mathsf{Q}$-martingale.

\subsection{Exchange option}

With the interpretation of $X$ and $Y$ as two stock price processes and
assuming an interest rate of $r=0$, we can define the price of a
European exchange option with strike $K\in\mathbb{R}_+$ (also known
as the
ratio of notionals) and maturity $T\in\mathbb{R}_+$ as
\[
E(K,T):=\mathbb{E}^\mathsf{P}(X_T-KY_T)^+.
\]
The corresponding price of the American option is given by
\[
A(K,T):=\sup_{\sigma\in\mathcal{T}_{0,T}}\mathbb{E}^\mathsf
{P}(X_\sigma -KY_\sigma)^+,
\]
where $\mathcal{T}_{0,T}$ is the set of all stopping times $\sigma$,
which take values in $[0,T]$. Let us define the last passage time $\rho
_K:=\sup \{t\geq0| Z_t=\frac{1}{K} \}$, where as usual
the supremum of the empty set is equal to zero. In the next theorem,
the prices of the European and American exchange option are expressed
in terms of the last passage time $\rho_K$ in the spirit of \cite{optionpricesprobabilities}.

%
\begin{thm}\label{last}
For all $K,T\geq0$, the prices of the European and American exchange
option are given by
\begin{eqnarray*}
E(K,T)&=& \mathbb{E}^\mathsf{Q} \bigl( (1-KZ_{\tau^X} )^+
\mathbh{1}_{\{\rho_K\leq
T<\tau^X\}} \bigr), 
\\
A(K,T) &=&
\mathbb{E}^\mathsf{Q} \bigl( (1-KZ_{\tau^X} )^+\mathbh{1}_{\{\rho_K\leq
T\}}
\bigr).
\end{eqnarray*}
\end{thm}

\begin{pf}
Assume $\sigma\in\mathcal{T}_{0,T}$. As seen above, $Z=\frac{Y}{X}$
is a nonnegative local $\mathsf{Q}$-martingale, thus a
supermartingale, which
converges almost surely to $Z_\infty=Z_{\tau^X}$. From Corollary~3.4
in \cite{drawdowns}, respectively Theorem~2.5 in \cite{optionpricesprobabilities} we have the identity
%
\setcounter{equation}{5}
\begin{equation}
\label{Doob}
\biggl(\frac{1}{K}-Z_\sigma \biggr)^+=
\mathbb{E}^\mathsf{Q} \biggl( \biggl(\frac{1}{K}-Z_{\tau^X}
\biggr)^+\mathbh{1}_{\{\rho_K\leq\sigma\}} \Big|\mathcal{F}_\sigma \biggr).
\end{equation}
Multiplying the above equation with the $\mathcal{F}_\sigma$-measurable random
variable  $K\mathbh{1}_{\{\tau^X>\sigma\}}$ and taking expectations
under $\mathsf{Q}
$ yields
\[
\mathbb{E}^\mathsf{Q} \bigl( (1-KZ_\sigma )^+\mathbh
{1}_{\{\tau^X>\sigma\}
} \bigr)= \mathbb{E}^\mathsf{Q} \bigl(
(1-KZ_{\tau^X} )^+\mathbh{1}_{\{\rho
_K\leq\sigma<\tau^X\}} \bigr).
\]
Changing the measure via $d\mathsf{P}|_{\mathcal{F}_\sigma}=\frac
{1}{X_\sigma}\,d\mathsf{Q}
|_{\mathcal{F}_\sigma}$, we obtain
%
\begin{eqnarray}
\mathbb{E}^\mathsf{P}
(X_\sigma-KY_\sigma )^+ &=& \mathbb {E}^\mathsf{P} \bigl(
\mathbh{1}_{\{\tau
^X>\sigma\}}X_\sigma (1-KZ_\sigma )^+ \bigr)
\nonumber
\\[-8pt]
\label{prob}\\[-8pt]
\nonumber
&=&  \mathbb {E}^\mathsf{Q} \bigl( (1-KZ_{\tau^X} )^+
\mathbh{1}_{\{\rho_K\leq\sigma
<\tau
^X\}} \bigr),
\end{eqnarray}
since $\mathbh{1}_{\{\tau^X>\sigma\}}=1$ $\mathsf{P}$-almost
surely. Taking $\sigma=T$ the formula for the European option is proven.
For the American option value we note that in the proof of Theorem~1.4
in \cite{BKX} it is shown that
\[
A(K,T)=\lim_{n\rightarrow\infty}\mathbb{E}^\mathsf{P}
\biggl(Y_{\tau^X_n\wedge T} \biggl(\frac{1}{Z_{\tau_n^X\wedge T}}-K \biggr)^+ \biggr)= \lim
_{n\rightarrow
\infty}\mathbb{E}^\mathsf{P} (X_{\tau_n^X\wedge T}-KY_{\tau
^X_n\wedge
T}
)^+.
\]
Setting $\sigma=\tau^X_n\wedge T$ in equality (\ref{prob}), it
follows that
\begin{eqnarray*}
&& A(K,T)\\
&& \qquad =\lim_{n\rightarrow\infty}\mathbb{E}^\mathsf{P}
(X_{\tau^X_n\wedge
T}-KY_{\tau^X_n\wedge T} )^+= \lim_{n\rightarrow\infty
}
\mathbb{E}^\mathsf{Q} \bigl( (1-KZ_{\tau^X} )^+\mathbh{1}_{\{\rho_K\leq\tau
^X_n\wedge
T<\tau^X\}}
\bigr)
\\
&& \qquad =\lim_{n\rightarrow\infty}\mathbb{E}^\mathsf{Q} \bigl(
(1-KZ_{\tau^X} )^+\mathbh{1} _{\{\rho_K\leq\tau^X_n\wedge T\}} \bigr) =
\mathbb{E}^\mathsf{Q} \bigl( (1-KZ_{\tau^X} )^+\mathbh
{1}_{\{\rho_K\leq\tau
^X\wedge T\}} \bigr)
\\
&& \qquad=\mathbb{E}^\mathsf{Q} \bigl( (1-KZ_{\tau^X} )^+\mathbh
{1}_{\{\rho_K\leq T\}
} \bigr),
\end{eqnarray*}
where the last equality follows from the fact that $Z_{\tau^X}=\frac
{1}{K}$ on $\{\rho_K>\tau^X\}= \{\rho_K=\infty\}$.
\end{pf}

%
\begin{rem}
Assume that $\mathsf{Q}(\tau^X<\infty)=1$, that is, $\mathbb{E}^\mathsf{P}X_t\stackrel{t\rightarrow\infty}{\longrightarrow}0$. If we take $Y\equiv1$ in
the above
theorem, we get the formula for the standard European call option
expressed as a function of the last passage time of $X$ as it can be
found in~\cite{YenYor} for the special case of Bessel processes or in
\cite{fromto}:
%
\begin{equation}
\label{lastprob}
E(K,T)=\mathsf{Q} \bigl(\rho_K\leq T<
\tau^X \bigr).
\end{equation}
More generally, for arbitrary $Y$ formula (\ref{lastprob}) is still
true, if (T) holds and $\mathsf{Q}(\tau^X<\infty)=1$.
\end{rem}

%
\begin{rem}
We can also express the price of a barrier exchange option in terms of
the last passage time of $Z$ at level $\frac{1}{K}$ as done in
Theorem~\ref{last} for exchange options without barriers. For example,
in the case of the Down-and-In exchange option we simply have to
multiply equation (\ref{Doob}) with the $\mathcal{F}_\sigma$-measurable
random variable $\mathbh{1}_{\{\hat{m}_\sigma^X\leq D\}}$.
\end{rem}

We now analyze a few special cases of Theorem~\ref{last} in more detail:

\begin{longlist}[(2)]
\item[(1)] $X$ is a true $\mathsf{P}$-martingale.

If $X$ is a true $\mathsf{P}$-martingale, the price process for $X$ exhibits
no asset price bubble. Then, regardless of whether the stock price
process $Y$ has an asset price bubble or not, we know that $\mathsf
{Q}$ is
locally equivalent to $\mathsf{P}$ and $\mathsf{Q}(\tau^X=\infty
)=1$. Therefore,
\[
E(K,T)=A(K,T)=\mathbb{E}^\mathsf{Q} \bigl( (1-{K} {Z_\infty } )^+
\mathbh{1}_{\{
\rho_K\leq T\}} \bigr)
\]
and the European and American exchange option values are equal. For
$Y\equiv1$, this formula is well known (cf. \cite{optionpricesprobabilities}).

\item[(2)] $Y$ is a true $\mathsf{P}$-martingale.

We recall from Lemma~\ref{XY} that in this case $Z_{\tau^X}=0$ on $\{
\tau^X<\infty\}$ $\mathsf{Q}$-almost surely. Denoting $\tau
^Z_0=\inf\{
t\geq0| Z_t=0\}$ this translates into $\mathsf{Q}(\tau^X=\tau
^Z_0)=1$, since
\[
\mathsf{Q}\bigl(\tau^Z_0<\tau^X\bigr)=\lim
_{n\rightarrow\infty}\mathsf {Q}\bigl(\tau^Z_0<\tau
_n^X\bigr)=\lim_{n\rightarrow\infty}
\mathbb{E}^\mathsf{P} (X_{\tau
_n^X}\mathbh{1}_{\{\tau
^Z_0<\tau_n^X\}} )=0.
\]
Therefore,
\begin{eqnarray*}
E(K,T)&=&\mathsf{Q} \bigl(\rho_K\leq T<\tau^Z_0
\bigr),
\\
A(K,T)&=&\mathsf{Q} \bigl(\rho_K\leq T\wedge\tau^X
\bigr)= \mathsf {Q} \bigl(\rho_K\leq T\wedge\tau^Z_0
\bigr)=\mathsf{Q} (\rho_K\leq T ),
\end{eqnarray*}
where the last equality follows from the fact that the last passage
time of the level $\frac{1}{K}$ by $Z$ cannot be greater than its
first hitting time of 0. Note that in this case the above formula for
$E(K,T)$ is similar to the one for the European call option given in
\cite{fromto}, Proposition~7; see also \cite{YenYor} for the case of
the reciprocal Bessel process of dimension greater than two.

Especially, the American option premium is equal to
\begin{eqnarray*}
&& A(K,T)-E(K,T)\\
&& \qquad=\mathsf{Q} (\rho_K\leq T )- \mathsf {Q} \bigl(\rho
_K\leq T<\tau^Z_0 \bigr)=\mathsf{Q} \bigl(
\rho_K\leq T,\tau ^Z_0\leq T \bigr)
\\
&&\qquad=\mathsf{Q}\bigl(\tau^Z_0\leq T\bigr)=\mathsf{Q}
\bigl(\tau^X\leq T\bigr)=\gamma_X(0,T),
\end{eqnarray*}
which is just the default of the local $\mathsf{P}$-martingale $X$ or, in
other words, the bubble of the stock $X$ between $0$ and $T$.

\item[(3)] $X$ and $Y$ are both strict local $\mathsf{P}$-martingales: An
example.

Let $X$ and $Y$ be the reciprocals of two independent BES(3)-processes
under~$\mathsf{P}$ and assume that $X_0=x\in\mathbb{R}_+$, while
$Y_0=1$. (Note that
this normalization is different from the previous one. However, since
the density of $X$ respectively $Y$ is explicitly known in this case,
we can do calculations directly under $\mathsf{P}$. This allows us to point
out some anomalies of the option value in the presence of strict local
martingales.)\

We apply the formula for the European call option value written on the
reciprocal BES(3)-process from Example~3.6 in \cite{CoxHobson} and
integrate over $Y$:
\begin{eqnarray*}
\hspace*{-4pt}E(K,T) &=& \int_0^\infty x \biggl[\Phi \biggl(
\frac{x-zK}{xzK\sqrt{T}} \biggr)-\Phi \biggl(-\frac{1}{x\sqrt{T}} \biggr)
\\
&&\hspace*{32pt}{}+\Phi \biggl(
\frac
{1}{x\sqrt{T}} \biggr)-\Phi \biggl(\frac{zK+x}{xzK\sqrt{T}} \biggr) \biggr]
\mathsf{P}(Y_T\in dz)
\\
&&{}-K\int_0^\infty z \biggl\{\Phi \biggl(
\frac{zK+x}{xzK\sqrt{T}} \biggr)-\Phi \biggl(\frac{zK-x}{xzK\sqrt{T}} \biggr)\\
&&\hspace*{54pt}{}+x\sqrt{T}
\biggl[\varphi \biggl(\frac{zK+x}{xzK\sqrt{T}} \biggr)-\varphi \biggl(\frac
{x-zK}{xzK\sqrt{T}}
\biggr) \biggr] \biggr\}\mathsf{P}(Y_T\in dz),
\end{eqnarray*}
where
\[
\mathsf{P}(Y_T\in dz)=\frac{1}{z^3}\frac{dz}{\sqrt{2\pi T}} \biggl(
\exp \biggl(-\frac{(1/z-1)^2}{2T} \biggr)-\exp \biggl(-\frac
{(1/z+1)^2}{2T} \biggr)
\biggr).
\]
Since $\mathbb{E}^\mathsf{P}X_T\stackrel{x\rightarrow\infty
}{\longrightarrow} \frac
{2}{\sqrt{2\pi T}}$ as shown in \cite{Hulley}, the option value
converges to a finite positive value as the initial stock price $X_0=x$
goes to infinity. Therefore, the convexity of the payoff function does
not carry over to the option value.
This anomaly for stock price bubbles has been noticed before by, for
example, \cite{CoxHobson,Hulley}. We refer for the economic intuition
of this phenomenon to \cite{Hulley}, where a detailed analysis of
stock and bond price bubbles modelled by the reciprocal BES(3)-process
is done.

Furthermore, recall that by Jensen's inequality the European exchange
option value is increasing in maturity if $X$ and $Y$ are true
martingales. However, in our example the option value is \textit{not}
increasing in maturity anymore: Indeed, because of $E(K,T)\leq\mathbb
{E}^\mathsf{P}
X_T\stackrel{T\rightarrow\infty}{\longrightarrow}0$ the option value
converges to zero as $T\rightarrow\infty$. Taking $Y\equiv1$, this behaviour
has been noticed before by, for example, \cite{CoxHobson,Hulley,MadanYor,PalP} and is also directly evident from the
representation of $E(K,T)$ in Theorem~\ref{last}.
\end{longlist}

\subsection{Real-world pricing}\label{real}

Here, we want to give another interpretation of Theorem~\ref{last}.
Note that from a mathematical point of view we have only assumed that
$X$ and $Y$ are strictly positive local $\mathsf{P}$-martingales for the
result. Above we have interpreted $\mathsf{P}$ as the risk-neutral probability
and $X,Y$ as two stock price processes. Now note that we have the
identity $(X-KY)^+=Y (\frac{1}{Z}-K )^+$. This motivates
the following alternative financial setting: we take $\mathsf{P}$ to
be the
historical probability and assume that also $\mathsf{P}(Y_0=1)=1$. Normalizing
the interest rate to be equal to zero, the process $S:=\frac{1}{Z}$
denotes the (discounted) stock price process, while $Y$ is a candidate
for the density of an equivalent local martingale measure (ELMM). Since
$Y$ and $X=YS$ are both strictly positive local $\mathsf
{P}$-martingales, they
are $\mathsf{P}$-super-martingales and cannot reach infinity under
$\mathsf{P}$.
Thus, $S=\frac{1}{Z}$ is also strictly positive under $\mathsf{P}$
and does
not attain infinity under $\mathsf{P}$ either.

As before, $X$ and $Y$ are both allowed to be either strict local or
true $\mathsf{P}$-martingales. While the question of whether $X=YS$ is
a true
martingale or not is related to the existence of a stock price bubble
as discussed earlier, the question of whether $Y$ is a strict local
martingale or not is connected to the absence of arbitrage. If $Y$ is a
uniformly integrable $\mathsf{P}$-martingale, an ELMM for $Z$ exists
and the
market satisfies NFLVR. However, as shown in \cite{Fern} and explained
in \cite{BKX}, even if $Y$ is only a strict local martingale, a
super-hedging strategy for any contingent claim written on $S$ exists.
Therefore, the ``normal'' call option pricing formulas
\[
E(K,T)=\mathbb{E}^\mathsf{P} \bigl(Y_T (S_T-K
)^+ \bigr), \qquad A(K,T)=\sup_{\sigma\in\mathcal{T}_{0,T}}\mathbb{E}^\mathsf{P}
\bigl(Y_\sigma (S_\sigma -K )^+ \bigr)
\]
are still reasonable when $Y$ is only a strict local martingale. This
pricing method is also known as ``real-world pricing,'' since we cannot
work under an ELMM directly, but must define the option value under the
real-world measure (cf. \cite{benchmark}). Note that if $Y$ is a true
martingale, we can define an ELMM $\mathsf{P}^*$ for $S$ on $\mathcal
{F}_T$ via $\mathsf{P}
^*|_{\mathcal{F}_T}=Y_T \cdot \mathsf{P}|_{\mathcal{F}_T}$ and the market
satisfies the NFLVR
property until time $T\in\mathbb{R}_+$. In this case, we obtain the usual
pricing formulas
\[
E(K,T)=\mathbb{E}^{\mathsf{P}^*}(S_T-K)^+ \quad \mbox{respectively} \quad  A(K,T)=\sup
_{\sigma\in
\mathcal{T}_{0,T}}\mathbb{E}^{\mathsf{P}^*}(S_\sigma-K)^+.
\]

Following \cite{Hulley}, we can interpret the situation when $Y$ is
only a strict local martingale as the existence of a bond price bubble
as opposed to the stock price bubble discussed above. This is motivated
by the fact that the real-world price of a zero-coupon bond is strictly
less than the (discounted) pay-off of one, if $Y$ is a strict local
martingale. Of course, it is possible to make a risk-free profit in
this case via an admissible trading strategy. From Theorem~\ref{last},
we have the following corollary.

%
\begin{cor}
For all $K,T\geq0$, the values of the European and American call
option under real-world pricing are given by
\begin{eqnarray*}
E(K,T)&=&\mathbb{E}^\mathsf{Q} \biggl( \biggl(1-\frac{K}{S_{\tau
^X}}
\biggr)^+\mathbh{1}_{\{
\rho_K^S\leq T<\tau^X\}} \biggr), \\
A(K,T) &=&  \mathbb{E}^\mathsf{Q}
\biggl( \biggl(1-\frac{K}{S_{\tau
^X}} \biggr)^+\mathbh{1}_{\{
\rho_K^S\leq T\}} \biggr)
\end{eqnarray*}
with $\rho_K^S=\sup\{t\geq0 | S_t=K\}$.
\end{cor}

From the above formulas for the European and American call options, it
can easily be seen that their values are generally different, unless
$X=YS$ is a true $\mathsf{P}$-martingale (in this case $\tau^X=\infty
$ $\mathsf{Q}
$-a.s.). Therefore, Merton's no early exercise theorem does not hold
anymore (cf. also \cite{BKX,CoxHobson,JPScomp,JPSincomp}).

Furthermore, note that we have the following formula for any bounded
stopping time $T$:
\[
E(K,T)=\mathbb{E}^\mathsf{P}(X_T-KY_T)^+=
\mathbb{E}^\mathsf{Q} (1-KZ_T )^+-\mathbb{E}^\mathsf{Q}
\bigl(\mathbh{1}_{\{\tau^X\leq T\}}(1-KZ_T)^+ \bigr),
\]
where the second term equals $\mathsf{Q}(\tau^X\leq T)$, if (T)
holds. For
$Y\equiv1$, this decomposition of the European call value is shown in
\cite{PalP}.

Now we show that also the asymptotic behaviour of the European and
American call option is unusual, when we allow $X$ and / or $Y$ to be
strict local $\mathsf{P}$-martingales. From the definition of the European
call option value, we easily see that
\[
\lim_{K\rightarrow0}E(K,T)=\mathbb{E}^\mathsf{P}(Y_TS_T)=
\mathbb {E}^\mathsf{P}X_T=\mathsf{Q}\bigl(\tau^X>T
\bigr), \qquad \lim_{K\rightarrow\infty}E(K,T)=0.
\]
Moreover, using the last passage time formula for the American call
derived above, it follows that
\[
\lim_{K\rightarrow0}A(K,T)=\lim_{K\rightarrow0}\mathsf{Q}
\bigl(\rho _K^S\leq T\bigr)=1,
\]
since $Z$ does not explode $\mathsf{Q}$-a.s., and hence $S$ is strictly
positive under $\mathsf{Q}$. Similarly, denoting $\rho^Z_{1/K}=\sup
 \{
t\geq0| Z_t=\frac{1}{K} \}$, we get
\begin{eqnarray*}
\lim_{K\rightarrow\infty}A(K,T)&=&\lim_{K\rightarrow\infty
}\mathsf{Q}
\bigl(\rho^S_K\leq T, S_{\tau^X}=\infty\bigr)=
\lim_{K\rightarrow\infty}\mathsf{Q}\bigl(\rho ^Z_{1/K}
\leq T, Z_{\tau^X}=0\bigr)
\\
&=& \mathsf{Q}(Z_{\tau^X}=Z_T=0) =\mathsf{Q}\bigl(T\geq
\tau^X, Z_{\tau^X}=0\bigr),
\end{eqnarray*}
which may be strictly positive and equals $\mathsf{Q}(T\geq\tau
^X)=\gamma
_X(0,T)$ under (T). For the asymptotics in $T$, we have
\begin{eqnarray*}
\lim_{T\rightarrow\infty} E(K,T)&=&\mathbb{E}^\mathsf{Q} \biggl(
\biggl(1-\frac{K}{S_{\tau
^X}} \biggr)^+\mathbh{1}_{\{\tau^X=\infty\}} \biggr),
\\
\lim_{T\rightarrow\infty} A(K,T)&=&\mathbb{E}^\mathsf{Q} \biggl(1-
\frac{K}{S_{\tau
^X}} \biggr)^+,
\end{eqnarray*}
and from the definition of the call option it is also clear that
\[
\lim_{T\rightarrow0}E(K,T)=\lim_{T\rightarrow0}A(K,T)=(1-K)^+.
\]

\subsubsection{American option premium under real-world pricing}

We keep the notation and interpretation introduced at the beginning of
Section~\ref{real}. However, we do not assume that $Z$ and/or $X$
are continuous anymore.

%
\begin{lem}\label{EA}
Let $h:\mathbb{R}_{++}\rightarrow\mathbb{R}_+$ be a Borel-measurable
function s.t.\break $\lim_{x\rightarrow\infty}\frac{h(x)}{x}=:\eta$ exists in $\mathbb
{R}_+$. Define $g:\mathbb{R}
_+\rightarrow\mathbb{R}_+$ via $g(x)=x\cdot h (\frac
{1}{x} )$ for $x>0$
and $g(0)=\eta$. We denote by $E(h,T)=\mathbb{E}^\mathsf
{P}(Y_Th(S_T))$ the value of
the European option with maturity $T$ and payoff function $h$ and by
$A(h,T)$ the value of the corresponding American option. Then
\[
E(h,T)=\mathbb{E}^\mathsf{Q}g(Z_T)-\mathbb{E}^\mathsf{Q}
\bigl(\mathbh{1}_{\{\tau^X\leq T\}}g (Z_{\tau^X} ) \bigr).
\]
Furthermore, if in addition $h$ is convex with $h(0)=0$, $h(x)\leq x$
for all $x\in\mathbb{R}_+$ and $\eta=1$, then
\[
A(h,T)=\mathbb{E}^\mathsf{Q}g ({Z_T} ).
\]
\end{lem}

\begin{pf}
For the European option value, we have
\begin{eqnarray*}
E(h,T) &=&  \mathbb{E}^\mathsf{P}\bigl(Y_T h(S_T)
\bigr)=\mathbb{E}^\mathsf {Q} \bigl(g(Z_T)
\mathbh{1}_{\{\tau^X>T\}
} \bigr)\\
&=&  \mathbb{E}^\mathsf{Q}g(Z_T)-
\mathbb{E}^\mathsf{Q} \bigl(\mathbh{1}_{\{\tau^X\leq T\}}g (Z_{\tau^X}
) \bigr).
\end{eqnarray*}
And for the American option value we get
\begin{eqnarray*}
A(h,T)&=&\lim_{n\rightarrow\infty}\mathbb{E}^\mathsf{P}
\bigl(Y_{T\wedge\tau
^X_n}h(S_{T\wedge\tau^X_n}) \bigr)= \lim_{n\rightarrow\infty
}
\mathbb{E}^\mathsf{Q} \biggl(Z_{T\wedge\tau^X_n}h \biggl(\frac{1}{Z_{T\wedge\tau^X_n}}
\biggr) \biggr)
\\
&=&\lim_{n\rightarrow\infty}\mathbb{E}^\mathsf{Q}g({Z_{T\wedge
\tau^X_n}})=
\mathbb{E}^\mathsf{Q} g({Z_{T\wedge\tau^X}})=\mathbb{E}^\mathsf{Q}g(Z_T),
\end{eqnarray*}
where the first equality is proven in \cite{BKX} under the above
stated assumptions on $h$ and the fourth equality follows by dominated
convergence since $g\leq1$ is a bounded and continuous function.
\end{pf}

Under the assumptions of Lemma~\ref{EA}, the American option premium
is thus equal to
\[
A(h,T)-E(h,T)=\mathbb{E}^\mathsf{Q} \bigl(\mathbh{1}_{\{\tau^X\leq
T\}}g
(Z_{\tau
^X} ) \bigr).
\]
Note that Lemma~\ref{EA} is a generalization of Theorem A1 in \cite{CoxHobson}. Indeed, if $Y$ is a~uniformly integrable $\mathsf{P}$-martingale
(i.e., NFLVR is satisfied), $Z_{\tau^X}=0$ on $\{\tau^X<\infty\}$ by
part 2(a) of Lemma~\ref{XY}. Thus,
\[
A(h,T)=E(h,T)+g(0)\cdot\mathsf{Q} \bigl(\tau^X\leq T \bigr)=E(h,T)+
\gamma_X(0,T).
\]

\section{Multivariate strictly positive (strict) local
martingales}\label{multi}

So far the measure $\mathsf{Q}$ defined in Theorem~\ref{komplett}
above is
only associated with the local $\mathsf{P}$-martingale $X$ in the
sense that
$X_{\tau_n^X}.\mathsf{P}|_{\mathcal{F}_{\tau^X_n}}=\mathsf
{Q}|_{\mathcal{F}_{\tau^X_n}}$ for all
$n\in\mathbb{N}$ and that $\frac{1}{X}$ is a true martingale under
$\mathsf{Q}$.
One may now naturally wonder whether, given two (or more) positive
local $\mathsf{P}$-martingales $X$ and $Y$, there exists a measure
$\mathsf{Q}$,
under which $\frac{1}{X}$ and $\frac{1}{Y}$ are both local (or even
true) martingales. Obviously, this is the case, if $X$~and $Y$ are
independent under $\mathsf{P}$. In this section, we will consider the case
where $X$~and $Y$ are continuous local $\mathsf{P}$-martingales, but not
necessarily independent.

%
\begin{thm}\label{m1}
Let $(\Omega,\mathcal{F},({\mathcal{F}}_t)_{t\geq0},\mathsf{P})$
be a filtered probability
space, where $(\mathcal{F}_t)_{t\geq0}$ is the right-continuous augmentation
of a standard system. Assume that $X$ and $Y$ are two strictly positive
continuous local $\mathsf{P}$-martingales with
$d\langle X\rangle_t=f_t\,dt$, $d\langle Y\rangle_t=g_t\,dt$ and
$d\langle X,Y\rangle_t=h_t\,dt$. Suppose that for all $t>0$, the
stochastic integral
\[
M_t=\int_0^t\frac{(f_s Y_s-h_sX_s)g_s}{Y_sX_s(f_sg_s-h^2_s)}\,dX_s+
\int_0^t\frac{(g_sX_s-h_sY_s)f_s}{Y_sX_s(f_sg_s-h^2_s)}\,dY_s
\]
is well-defined. Denote by $\tau$ the explosion time of $\mathcal{E}(M)$.
Then there exists a measure $\mathsf{Q}$ on $\mathcal{F}_\infty$,
under which $\frac
{1}{\tilde{X}}$ and $\frac{1}{\tilde{Y}}$ defined via
\begin{eqnarray*}
\tilde{X}_t&=&X_t\mathbh{1}_{\{t<\tau\}}+\liminf
_{s\rightarrow\tau
,s<\tau,s\in\mathbb{Q}
}X_s\mathbh{1}_{\{\tau\leq t<\infty\}},
\\
\tilde{Y}_t&=&Y_t\mathbh{1}_{\{t<\tau\}}+\liminf
_{s\rightarrow\tau
,s<\tau,s\in\mathbb{Q}
}Y_s\mathbh{1}_{\{\tau\leq t<\infty\}}
\end{eqnarray*}
are both continuous nonnegative local $\mathsf{Q}$-martingales and
$d\mathsf{P}|_{\mathcal{F}
_t}=\frac{1}{\mathcal{E}(M)_t}\times \mathbh{1}_{\{t<\tau\}}\,d\mathsf
{Q}|_{\mathcal{F}_t}$ for all
$t\geq0$. 
\end{thm}

\begin{pf}
The stochastic exponential $\mathcal{E}(M)$ is a
continuous local $\mathsf{P}$-martin\-gale with localizing sequence
\[
\tau_n:=\inf\bigl\{t\geq0\dvtx \mathcal{E}(M)_t>n\bigr\}
\wedge n.
\]
We define a consistent family of probability measures $\mathsf{Q}_n$
on $\mathcal{F}
_{\tau_n}$ by
\[
\frac{d\mathsf{Q}_n}{d\mathsf{P}}\bigg|_{\mathcal{F}_{\tau_n}}=\mathcal {E}(M)_{\tau_n},\qquad  n\in
\mathbb{N}.
\]
Using the same trick as in the proof of Theorem~\ref{komplett}, we
restrict each measure~$\mathsf{Q}_n$ to $\mathcal{F}_{\tau_n-}$.
Since $(\mathcal{F}_{\tau
_n-})_{n\in\mathbb{N}}$ is a standard system by Lemma~\ref{standard}, there
exists a unique measure $\mathsf{Q}$ on $\mathcal{F}_{\tau-}$, such
that $\mathsf{Q}|_{\mathcal{F}
_{\tau_n}}=\mathsf{Q}_n$ for all $n\in\mathbb{N}$. For any stopping
time $S$ and
$A\in\mathcal{F}_S$, we get
\[
\mathsf{Q}(S<\tau_n,A)=\mathbb{E}^\mathsf{P} \bigl(\mathcal
{E}(M)_{S\wedge\tau_n}\mathbh{1}_{\{
S<\tau_n,A\}} \bigr)=\mathbb{E}^\mathsf{P}
\bigl(\mathcal {E}(M)_{S}\mathbh{1}_{\{S<\tau
_n,A\}} \bigr).
\]
Taking $n\rightarrow\infty$ results in
\[
\mathsf{Q}(S<\tau,A)=\mathbb{E}^\mathsf{P} \bigl(\mathcal
{E}(M)_S\mathbh{1}_{\{S<\infty,A\}
} \bigr).
\]
It follows that $\mathsf{P}$ is locally absolutely continuous with
respect to
$\mathsf{Q}$ before $\tau$. We choose an arbitrary extension of
$\mathsf{Q}$ from
$\mathcal{F}_{\tau-}$ to $\mathcal{F}_\infty$ as discussed on page~\pageref{extension}. Next, according to Girsanov's theorem applied on $\mathcal{F}
_{\tau_n}$,
\begin{eqnarray*}
N_{t\wedge\tau_n}&:=&X_t^{\tau_n}-\bigl\langle
M^{\tau_n},X^{\tau
_n}\bigr\rangle_t\\
&\,\,=&  X^{\tau_n}_t-
\int_0^{t\wedge\tau_n}\frac{(f_s
Y_s-h_sX_s)g_s}{Y_sX_s(f_sg_s-h^2_s)}\,d\langle X
\rangle_s
\\
&&{}-\int_0^{t\wedge\tau_n}\frac
{(g_sX_s-h_sY_s)f_s}{Y_sX_s(f_sg_s-h^2_s)}\,d\langle X,Y
\rangle_s
\\
&\,\,=&X^{\tau_n}_t-\int_0^{t\wedge\tau_n}
\frac{(f_s
Y_s-h_sX_s)g_sf_s+(g_sX_s-h_sY_s)f_sh_s}{Y_sX_s(f_sg_s-h^2_s)}\,ds \\
&\,\,=& X^{\tau_n}_t-\int_0^{t\wedge\tau_n}
\frac{f_s}{X_s}\,ds
\end{eqnarray*}
is a local $\mathsf{Q}$-martingale. We apply It\^o's formula:
\begin{eqnarray*}
\frac{1}{X_{t \wedge{\tau_n}}}&=&\frac{1}{X_0}-\int_0^{t\wedge
\tau_n}
\frac{dX_s}{X_s^2}+\int_0^{t\wedge\tau_n}
\frac{d\langle
X\rangle_s}{X_s^3}
\\
&=&\frac{1}{X_0}-\int_0^{t\wedge\tau_n}
\frac
{dN_s}{X_s^2}-\int_0^{t\wedge\tau_n}
\frac{f_s}{X^3_s}\,ds+\int_0^{t\wedge\tau_n}
\frac{f_s}{X_s^3}\,ds=\frac{1}{X_0}-\int_0^{t\wedge\tau_n}
\frac{dN_s}{X_s^2}.
\end{eqnarray*}
Thus, $\frac{1}{X^{\tau_n}}$ is a local $\mathsf{Q}$-martingale for all
$n\in\mathbb{N}$. Since $\frac{1}{X}$ is continuous, $(\tau
_m^{1/X})_{m\in
\mathbb{N}}$ is a localizing sequence for $\frac{1}{X^{\tau_n}}$ on
$(\Omega
,\mathcal{F}_{\tau_n},\mathsf{Q})$ for all $n\in\mathbb{N}$, where
\[
\tau_m^{1/X}:= \inf \biggl\{t\geq0\dvtx \frac{1}{X_t}> m
\biggr\}\wedge m, \qquad \tau^{1/X}:=\lim_{m\rightarrow\infty}
\tau^{1/X}_m.
\]
Moreover, we have
\[
\mathsf{Q}\bigl(\tau^{1/X}<\tau\bigr)=\lim_{n\rightarrow\infty}
\mathsf {Q}\bigl(\tau^{1/X}<\tau_n\bigr)= \lim
_{n\rightarrow\infty}\mathbb{E}^\mathsf{P} \bigl(\mathcal
{E}(M)_{\tau_n}\mathbh{1}_{\{\tau
^{1/X}<\tau_n\}} \bigr)=0,
\]
because $X$ is strictly positive under $\mathsf{P}$. Since a process
which is
locally a local martingale is a local martingale itself, we conclude
that $\frac{1}{X}$ is a positive local $\mathsf{Q}$-martingale up to time
$\tau$ with localizing sequence $(\tau_n\wedge\tau_n^{1/X})_{n\in
\mathbb{N}}$. Especially, $\lim_{n\rightarrow\infty}X_{\tau
_n}=\lim_{n\rightarrow\infty
}X_{\tau_n\wedge\tau_n^{1/X}}$ exists $\mathsf{Q}$-almost surely. Thus,
$\frac{1}{\tilde{X}}$ is a continuous positive $\mathsf{Q}$-super-martingale
and $\tau^{1/X}_n\rightarrow\infty$ $\mathsf{Q}$-almost surely. Therefore,
\[
1\geq\mathbb{E}^\mathsf{Q} \biggl(\frac{1}{\tilde{X}_{\tau
_n^{1/X}}} \biggr)=\lim
_{m\rightarrow\infty}\mathbb{E}^\mathsf{Q} \biggl(\frac{1}{\tilde
{X}_{\tau_n^{1/X}\wedge
\tau_m}}
\biggr)\geq \lim_{m\rightarrow\infty}\mathbb{E}^\mathsf{Q} \biggl(
\frac
{1}{\tilde{X}_{\tau
_m^{1/X}\wedge\tau_m}} \biggr)=1,
\]
where the two inequalities follow by the super-martingale property.
Hence, $\frac{1}{\tilde{X}}$ is a local $\mathsf{Q}$-martingale.

%
%
%
%
%
%
For $\frac{1}{\tilde{Y}}$, the claim follows by analogous calculations.
\end{pf}

But are $\frac{1}{\tilde{X}}$ and $\frac{1}{\tilde{Y}}$ in the
setting of Theorem~\ref{m1} actually true $\mathsf{Q}$-martingales or just
local $\mathsf{Q}$-martingales? In general, there does not seem to be
an easy
answer to this question. However, if $X$ (resp. $Y$) is a homogeneous
diffusion, one can show the following extension of the above theorem.

%
\begin{lem}\label{ml}
In the setting of Theorem~\ref{m1} assume that $X$ follows the 
$\mathsf{P}$-dynamics
\[
dX_t=\sigma(X_t)\,dB_t
\]
for some $\mathsf{P}$-Brownian motion $B$, where $\sigma(\cdot)$ is locally
bounded and bounded away from zero on $(0,\infty)$ and $\sigma(0)=0$.
Then $\frac{1}{\tilde{X}}$ is a $\mathsf{Q}$-martingale, where the
measure~$\mathsf{Q}$ is constructed in Theorem~\ref{m1}.
\end{lem}

\begin{pf}
Note that, with the notation used in the proof of Theorem~\ref{m1}, up
to time $\tau$ the process $N$ follows the dynamics
\[
dN_t=\sigma(X_t)\,dB_t^\mathsf{Q},
\]
where
\[
B_t^\mathsf{Q}:=B_t-\int_0^t
\frac{\sigma(X_s)}{X_s}\,ds
\]
is a $\mathsf{Q}$-Brownian motion on $[0,\tau)$ by L\'evy's theorem. Hence,
the $\mathsf{Q}$-dynamics of~$\frac{1}{X}$ up to time $\tau$ are
given by
%
\begin{equation}
\label{dyn} d \biggl(\frac{1}{X_t} \biggr)=-\frac{\sigma(X_t)}{X_t^2}\,dB_t^\mathsf{Q}
=:\overline{\sigma} \biggl(\frac{1}{X_t} \biggr)\,dB_t^\mathsf{Q}
\end{equation}
and we are in a situation similar to Example~\ref{difex}. Especially,
$\frac{1}{\tilde{X}}$ is a stopped homogeneous diffusion under
$\mathsf{Q}$.
Recall that since $X$ is strictly positive under $\mathsf{P}$, we must have
\[
\int_0^1\frac{x}{\sigma^2(x)}\,dx=\infty.
\]
But any diffusion on an auxiliary probability space with the dynamics
described in (\ref{dyn}) satisfies
\[
\int_1^\infty\frac{x}{\overline{\sigma}^2(x)}\,dx=\int
_0^1\frac
{y}{\sigma^2(y)}\,dy=\infty
\]
and is hence a true martingale by the criterion of \cite{DelbaenShirakawa}, cf. also Example~\ref{difex}. Naturally, any
stopped diffusion with the same dynamics is a martingale as well. Since
the fact whether $\frac{1}{\tilde{X}}$ is a true martingale or not
only depends on its distributional properties, we may therefore
conclude that $\frac{1}{\tilde{X}}$ is indeed a $\mathsf{Q}$-martingale.
\end{pf}

%
\begin{rem}
Theorem~\ref{m1} deals with two strictly positive local $\mathsf{P}$-martin\-gales. It is, however, obvious that one can get a similar
result for $n\geq2$ strictly positive local $\mathsf{P}$-martingales. Also
note that the construction in Theorem~\ref{m1} is only possible if the
local quadratic covariation matrix of the local $\mathsf
{P}$-martingales is
sufficiently nondegenerate. Moreover, it is interesting that the
statement of Lemma~\ref{ml} contains no further restrictions on the
stochastic behaviour of $Y$.
\end{rem}

We briefly want to describe a different approach focusing on
``conformal local martingales'' in $\mathbb{R}^d, d>2$, which is
dealt with
in \cite{PalP}.

%
\begin{df}
A continuous local martingale $X$, taking values in $\mathbb{R}^d$, is called
a conformal local martingale on $(\Omega,\mathcal{F},(\mathcal
{F}_t)_{t\geq0},\mathsf{P})$,
if $\langle X^i,X^j\rangle=\langle X^1\rangle\mathbh{1}_{\{i=j\}}$
$\mathsf{P}
$-almost surely for all $1\leq i,j\leq d$.
\end{df}

In \cite{PalP}, the authors make the restriction that the conformal
local martingale does not enter some compact neighborhood of the origin
in $\mathbb{R}^d$. Using simple localization arguments as in
Theorem~\ref{komplett} above, one can get rid off this assumption which seems
somehow inappropriate when dealing with stock price processes. This
yields the following extended version of Lemma~12 in \cite{PalP}. We
denote by $|\cdot|$ the Euclidean norm in $\mathbb{R}^d$.

%
\begin{thm}
Let $(\Omega,\mathcal{F},({\mathcal{F}}_t)_{t\geq0},\mathsf{P})$
be a filtered probability
space such that $(\mathcal{F}_t)_{t\geq0}$ is the right-continuous
augmentation of a standard system.
For $d>2$, let $X=(X^1,\ldots,X^d)$ be a conformal local $\mathsf{P}
$-martingale. Suppose that $X_0=x_0$ with $|x_0|=1$ and define
$\tau:=\inf\{t\geq0|  |X_t|=0\}$.
Then there exists a measure $\mathsf{Q}$ on $\mathcal{F}_{\infty}$,
such that $\mathsf{Q}
|_{\mathcal{F}_t}\gg\mathsf{P}|_{\mathcal{F}_t}$ for all $t\geq0$
and such that
\[
Y_t:= %
\cases{\ds\frac{X_t}{|X_t|^2}, &\quad  $t<\tau$,
\vspace*{3pt}\cr
\ds\liminf
_{s\rightarrow
\tau
,s<\tau,s\in\mathbb{Q}} \frac{X_s}{|X_s|^2}, &\quad  $t\geq\tau$}
\]
is a conformal uniformly-integrable $\mathsf{Q}$-martingale.
\end{thm}

\begin{pf}
Note that $\mathsf{P}(\tau<\infty)=0$ by Knight's theorem because a standard
$d$-dimensional Brownian motion does not return to the origin almost
surely for $d>2$. We define the stopping times $\tau_n:=\inf\{t\geq
0\dvtx |X_t|\leq\frac{1}{n}\}$. As in Lemma~11 in \cite{PalP}, it
follows that $ (|X_{t\wedge\tau_n}|^{2-d} )_{t\geq0}$ is
a uniformly integrable $\mathsf{P}$-martingale for all $n\in\mathbb
{N}$, because
$|\cdot|^{2-d}$ is harmonic. We define a consistent family of
probability measures $\mathsf{Q}_n$ on $\mathcal{F}_{\tau_n}$ by
\[
\frac{d\mathsf{Q}_n}{d\mathsf{P}}\bigg|_{\mathcal{F}_{\tau_n}}=|X_{\tau
_n}|^{2-d},\qquad  n\in
\mathbb{N}.
\]
Using the same trick as in the proof of Theorem~\ref{komplett}, we
restrict each measure $\mathsf{Q}_n$ to $\mathcal{F}_{\tau_n-}$.
Since $(\mathcal{F}_{\tau
_n-})_{n\in\mathbb{N}}$ is a standard system, there exists a unique measure
$\mathsf{Q}$ on $\mathcal{F}_{\tau-}$, such that $\mathsf
{Q}|_{\mathcal{F}_{\tau_n}}=\mathsf{Q}_n$ for all
$n\in\mathbb{N}$. For any stopping time $S$, we thus get
\[
\mathsf{Q}(S<\tau_n)=\mathbb{E}^\mathsf{P}
\bigl(|X_{\tau
_n}|^{2-d}\mathbh{1}_{\{S<\tau_n\}
} \bigr)=
\mathbb{E}^\mathsf{P} \bigl(|X_{S}|^{2-d}
\mathbh{1}_{\{
S<\tau_n\}} \bigr).
\]
Choosing $S=t<\infty, A\in\mathcal{F}_t$ and taking $n\rightarrow
\infty$ results in
\[
\mathsf{Q}\bigl(A\cap\{t<\tau\}\bigr)=\mathbb{E}^\mathsf{P}
\bigl(|X_t|^{2-d}\mathbh{1}_A \bigr).
\]
Therefore, $\mathsf{P}$ is locally absolutely continuous to $\mathsf
{Q}$ before $\tau$. As explained on page~\pageref{extension} there exists an extension
of $\mathsf{Q}$ from $\mathcal{F}_{\tau-}$ to $\mathcal{F}_\infty
$, which we also denote by
$\mathsf{Q}$.

From Lemma~12 in \cite{PalP}, we know that $\frac{X_{t\wedge\tau
_n}}{|X_{t\wedge\tau_n}|^2}$ is a conformal $\mathsf
{Q}_n$-martingale. Furthermore,
\[
\Bigl(\mathbb{E}^\mathsf{Q}\sup_{t<\tau}
|Y^i_t| \Bigr)^2\leq \mathbb{E}^\mathsf{Q}
\sup_{t<\tau} |Y^i_t|^2\leq1,\qquad
1\leq i\leq d.
\]
Thus, $Y$ is a continuous uniformly integrable $\mathsf{Q}$-martingale by
Exercise 1.48 in Chapter IV of \cite{RevuzYor}. Clearly, $Y$ is also conformal.
\end{pf}

\begin{appendix}
\section*{Appendix: Condition $(P)$}\label{AppP}
In Theorem~\ref{thm1}, we mentioned condition $(P)$, which was
introduced in Definition~4.1 in \cite{newkind} following \cite{para}
as follows.

%
\begin{df}
Let $(\Omega, \mathcal{F}, (\mathcal{F}_t )_{t\geq0} )$ be a
filtered measurable
space, such that $\mathcal{F}$ is the $\sigma$-algebra generated by $(\mathcal{F}_t)_{t\geq0} \dvtx  \mathcal{F}=
\bigvee_{t\geq0}\mathcal{F}_t$.
We shall say that the property $(P)$ holds if
and only if $(\mathcal{F}_t )_{t\geq0}$ enjoys the following conditions:
\begin{itemize}
\item For all $t \geq0$, $\mathcal{F}_t$ is generated by a countable
number of sets.
\item For all $t \geq0$, there exists a Polish space $\Omega_t$, and
a surjective map $\pi_t$ from $\Omega$ to~$\Omega_t$,
such that $\mathcal{F}_t$ is the $\sigma$-algebra of the inverse
images by
$\pi_t$ of Borel sets in~$\Omega_t$, and
such that for all $B\in\mathcal{F}_t$, $\omega\in\Omega$, $\pi
_t(\omega
)\in\pi_t (B)$ implies $\omega\in B$.
\item If $(\omega_n )_{ n\geq0}$ is a sequence of elements of $\Omega
$ such that for all $N\geq0$,
\[
\bigcap_{n\geq0}^N A_n (
\omega_n ) \neq\varnothing,
\]
where $A_n (\omega_n )$ is the intersection of the sets in $\mathcal{F}_n$
containing $\omega_n$, then
\[
\bigcap_{n\geq0}^\infty A_n (
\omega_n ) \neq\varnothing.
\]
\end{itemize}
\end{df}
\end{appendix}

%





\printaddresses

\begin{thebibliography}{37}

\bibitem{BKX}
%
\begin{barticle}[mr]
\bauthor{\bsnm{Bayraktar},~\bfnm{Erhan}\binits{E.}},
\bauthor{\bsnm{Kardaras},~\bfnm{Constantinos}\binits{C.}} \AND
\bauthor{\bsnm{Xing},~\bfnm{Hao}\binits{H.}}
(\byear{2012}).
\btitle{Strict local martingale deflators and valuing {A}merican
call-type options}.
\bjournal{Finance Stoch.}
\bvolume{16}
\bpages{275--291}.
\bid{doi={10.1007/s00780-011-0155-y}, issn={0949-2984}, mr={2903625}}
\end{barticle}
%
\bptok{imsref}%
\endbibitem

\bibitem{Carr}
%
\begin{barticle}[mr]
\bauthor{\bsnm{Carr},~\bfnm{Peter}\binits{P.}},
\bauthor{\bsnm{Fisher},~\bfnm{Travis}\binits{T.}} \AND
\bauthor{\bsnm{Ruf},~\bfnm{Johannes}\binits{J.}}
(\byear{2014}).
\btitle{On the hedging of options on exploding exchange rates}.
\bjournal{Finance Stoch.}
\bvolume{18}
\bpages{115--144}.
\bid{doi={10.1007/s00780-013-0218-3}, issn={0949-2984}, mr={3146489}}
\end{barticle}
%
\bptok{imsref}%
\endbibitem

\bibitem{drawdowns}
%
\begin{barticle}[mr]
\bauthor{\bsnm{Cheridito},~\bfnm{Patrick}\binits{P.}},
\bauthor{\bsnm{Nikeghbali},~\bfnm{Ashkan}\binits{A.}} \AND
\bauthor{\bsnm{Platen},~\bfnm{Eckhard}\binits{E.}}
(\byear{2012}).
\btitle{Processes of class sigma, last passage times, and drawdowns}.
\bjournal{SIAM J. Financial Math.}
\bvolume{3}
\bpages{280--303}.
\bid{doi={10.1137/09077878X}, issn={1945-497X}, mr={2968035}}
\end{barticle}
%
\bptok{imsref}%
\endbibitem

\bibitem{Chybi}
%
\begin{bincollection}[mr]
\bauthor{\bsnm{Chybiryakov},~\bfnm{Oleksandr}\binits{O.}}
(\byear{2007}).
\btitle{It\^o's integrated formula for strict local martingales with jumps}.
In \bbooktitle{S\'eminaire de {P}robabilit\'es {XL}}.
\bseries{Lecture Notes in Math.}
\bvolume{1899}
\bpages{375--388}.
\bpublisher{Springer},
\blocation{Berlin}.
\bid{doi={10.1007/978-3-540-71189-6_20}, mr={2409017}}
\end{bincollection}
%
\bptok{imsref}%
\endbibitem

\bibitem{CoxHobson}
%
\begin{barticle}[mr]
\bauthor{\bsnm{Cox},~\bfnm{Alexander~M.~G.}\binits{A.~M.~G.}} \AND
\bauthor{\bsnm{Hobson},~\bfnm{David~G.}\binits{D.~G.}}
(\byear{2005}).
\btitle{Local martingales, bubbles and option prices}.
\bjournal{Finance Stoch.}
\bvolume{9}
\bpages{477--492}.
\bid{doi={10.1007/s00780-005-0162-y}, issn={0949-2984}, mr={2213778}}
\end{barticle}
%
\bptok{imsref}%
\endbibitem

\bibitem{DSbessel}
%
\begin{barticle}[mr]
\bauthor{\bsnm{Delbaen},~\bfnm{F.}\binits{F.}} \AND
\bauthor{\bsnm{Schachermayer},~\bfnm{W.}\binits{W.}}
(\byear{1995}).
\btitle{Arbitrage possibilities in {B}essel processes and their
relations to local martingales}.
\bjournal{Probab. Theory Related Fields}
\bvolume{102}
\bpages{357--366}.
\bid{doi={10.1007/BF01192466}, issn={0178-8051}, mr={1339738}}
\end{barticle}
%
\bptok{imsref}%
\endbibitem

\bibitem{NFLVRlocal}
%
\begin{barticle}[mr]
\bauthor{\bsnm{Delbaen},~\bfnm{F.}\binits{F.}} \AND
\bauthor{\bsnm{Schachermayer},~\bfnm{W.}\binits{W.}}
(\byear{1998}).
\btitle{The fundamental theorem of asset pricing for unbounded
stochastic processes}.
\bjournal{Math. Ann.}
\bvolume{312}
\bpages{215--250}.
\bid{doi={10.1007/s002080050220}, issn={0025-5831}, mr={1671792}}
\end{barticle}
%
\bptok{imsref}%
\endbibitem

\bibitem{DelbaenShirakawa}
%
\begin{barticle}[auto:STB|2014/08/04|07:23:14]
\bauthor{\bsnm{Delbaen},~\bfnm{F.}\binits{F.}} \AND
\bauthor{\bsnm{Shirakawa},~\bfnm{H.}\binits{H.}}
(\byear{2002}).
\btitle{No arbitrage condition for positive diffusion price processes}.
\bjournal{Asia-Pac. Financ. Mark.}
\bvolume{9}
\bpages{159--168}.
\end{barticle}
%
\bptok{imsref}%
\endbibitem

\bibitem{dell}
%
\begin{bincollection}[mr]
\bauthor{\bsnm{Dellacherie},~\bfnm{C.}\binits{C.}}
(\byear{1969}).
\btitle{Ensembles al\'eatoires. {I}}.
In \bbooktitle{S\'eminaire de {P}robabilit\'es, III ({U}niv.
{S}trasbourg, 1967/68)}
\bpages{97--114}.
\bpublisher{Springer},
\blocation{Berlin}.
\bid{mr={0258120}}
\end{bincollection}
%
\bptok{imsref}%
\endbibitem

\bibitem{ELYtails}
%
\begin{bincollection}[mr]
\bauthor{\bsnm{Elworthy},~\bfnm{K.~D.}\binits{K.~D.}},
\bauthor{\bsnm{Li},~\bfnm{X.~M.}\binits{X.~M.}} \AND
\bauthor{\bsnm{Yor},~\bfnm{M.}\binits{M.}}
(\byear{1997}).
\btitle{On the tails of the supremum and the quadratic variation of
strictly local martingales}.
In \bbooktitle{S\'eminaire de {P}robabilit\'es, {XXXI}}.
\bseries{Lecture Notes in Math.}
\bvolume{1655}
\bpages{113--125}.
\bpublisher{Springer},
\blocation{Berlin}.
\bid{doi={10.1007/BFb0119298}, mr={1478722}}
\end{bincollection}
%
\bptok{imsref}%
\endbibitem

\bibitem{ELYradial}
%
\begin{barticle}[mr]
\bauthor{\bsnm{Elworthy},~\bfnm{K.~D.}\binits{K.~D.}},
\bauthor{\bsnm{Li},~\bfnm{Xue-Mei}\binits{X.-M.}} \AND
\bauthor{\bsnm{Yor},~\bfnm{M.}\binits{M.}}
(\byear{1999}).
\btitle{The importance of strictly local martingales; applications to
radial {O}rnstein--{U}hlenbeck processes}.
\bjournal{Probab. Theory Related Fields}
\bvolume{115}
\bpages{325--355}.
\bid{doi={10.1007/s004400050240}, issn={0178-8051}, mr={1725406}}
\end{barticle}
%
\bptok{imsref}%
\endbibitem

\bibitem{Ershov}
%
\begin{barticle}[auto]
\bauthor{\bsnm{Er{\v{s}}ov},~\bfnm{M.~P.}\binits{M.~P.}}
(\byear{1975}).
\btitle{Extension of measures and stochastic equations}.
\bjournal{Theory Probab. Appl.}
\bvolume{19}
\bpages{431--444}.
\end{barticle}
%
\bptok{imsref}%
\endbibitem

\bibitem{Fern}
%
\begin{bincollection}[auto:STB|2014/08/04|07:23:14]
\bauthor{\bsnm{Fernholz},~\bfnm{E.~R.}\binits{E.~R.}} \AND
\bauthor{\bsnm{Karatzas},~\bfnm{I.}\binits{I.}}
(\byear{2009}).
\btitle{Stochastic portfolio theory: An overview}.
In \bbooktitle{Special Volume: Mathematical Modeling and Numerical
Methods in Finance}
(\beditor{\bfnm{P.~G.}\binits{P.~G.}~\bsnm{Ciarlet}} et al., eds).
\bseries{Handbook of Numerical Analysis}
\bvolume{XV}
\bpages{89--167}.
\bpublisher{Elsevier},
\blocation{Oxford}.
\end{bincollection}
%
\bptok{imsref}%
\endbibitem

\bibitem{FernKar}
%
\begin{barticle}[mr]
\bauthor{\bsnm{Fernholz},~\bfnm{Daniel}\binits{D.}} \AND
\bauthor{\bsnm{Karatzas},~\bfnm{Ioannis}\binits{I.}}
(\byear{2010}).
\btitle{On optimal arbitrage}.
\bjournal{Ann. Appl. Probab.}
\bvolume{20}
\bpages{1179--1204}.
\bid{doi={10.1214/09-AAP642}, issn={1050-5164}, mr={2676936}}
\end{barticle}
%
\bptok{imsref}%
\endbibitem

\bibitem{exit}
%
\begin{barticle}[mr]
\bauthor{\bsnm{F{\"o}llmer},~\bfnm{Hans}\binits{H.}}
(\byear{1972}).
\btitle{The exit measure of a supermartingale}.
\bjournal{Z. Wahrsch. Verw. Gebiete}
\bvolume{21}
\bpages{154--166}.
\bid{mr={0309184}}
\end{barticle}
%
\bptok{imsref}%
\endbibitem

\bibitem{FollmerProtter}
%
\begin{barticle}[mr]
\bauthor{\bsnm{F{\"o}llmer},~\bfnm{Hans}\binits{H.}} \AND
\bauthor{\bsnm{Protter},~\bfnm{Philip}\binits{P.}}
(\byear{2011}).
\btitle{Local martingales and filtration shrinkage}.
\bjournal{ESAIM Probab. Stat.}
\bvolume{15}
\bpages{S25--S38}.
\bid{doi={10.1051/ps/2010023}, issn={1292-8100}, mr={2817343}}
\end{barticle}
%
\bptok{imsref}%
\endbibitem

\bibitem{Hulley}
%
\begin{bincollection}[mr]
\bauthor{\bsnm{Hulley},~\bfnm{Hardy}\binits{H.}}
(\byear{2010}).
\btitle{The economic plausibility of strict local martingales in
financial modelling}.
In \bbooktitle{Contemporary Quantitative Finance}
\bpages{53--75}.
\bpublisher{Springer},
\blocation{Berlin}.
\bid{doi={10.1007/978-3-642-03479-4_4}, mr={2732840}}
\end{bincollection}
%
\bptok{imsref}%
\endbibitem

\bibitem{Jacod}
%
\begin{bbook}[mr]
\bauthor{\bsnm{Jacod},~\bfnm{Jean}\binits{J.}}
(\byear{1979}).
\btitle{Calcul Stochastique et Probl\`emes de Martingales}.
\bseries{Lecture Notes in Math.}
\bvolume{714}.
\bpublisher{Springer},
\blocation{Berlin}.
\bid{mr={0542115}}
\end{bbook}
%
\bptok{imsref}%
\endbibitem

\bibitem{JaSh}
%
\begin{bbook}[mr]
\bauthor{\bsnm{Jacod},~\bfnm{Jean}\binits{J.}} \AND
\bauthor{\bsnm{Shiryaev},~\bfnm{Albert~N.}\binits{A.~N.}}
(\byear{2003}).
\btitle{Limit Theorems for Stochastic Processes},
\bedition{2nd} ed.
\bseries{Grundlehren der Mathematischen Wissenschaften [Fundamental
Principles of Mathematical Sciences]}
\bvolume{288}.
\bpublisher{Springer},
\blocation{Berlin}.
\bid{doi={10.1007/978-3-662-05265-5}, mr={1943877}}
\end{bbook}
%
\bptok{imsref}%
\endbibitem

\bibitem{JPffbubbles}
%
\begin{barticle}[mr]
\bauthor{\bsnm{Jarrow},~\bfnm{Robert~A.}\binits{R.~A.}} \AND
\bauthor{\bsnm{Protter},~\bfnm{Philip}\binits{P.}}
(\byear{2009}).
\btitle{Forward and futures prices with bubbles}.
\bjournal{Int. J. Theor. Appl. Finance}
\bvolume{12}
\bpages{901--924}.
\bid{doi={10.1142/S0219024909005518}, issn={0219-0249}, mr={2574496}}
\end{barticle}
%
\bptok{imsref}%
\endbibitem

\bibitem{JPScomp}
%
\begin{bincollection}[mr]
\bauthor{\bsnm{Jarrow},~\bfnm{Robert~A.}\binits{R.~A.}},
\bauthor{\bsnm{Protter},~\bfnm{Philip}\binits{P.}} \AND
\bauthor{\bsnm{Shimbo},~\bfnm{Kazuhiro}\binits{K.}}
(\byear{2007}).
\btitle{Asset price bubbles in complete markets}.
In \bbooktitle{Advances in Mathematical Finance}.
\bseries{Appl. Numer. Harmon. Anal.}
\bpages{97--121}.
\bpublisher{Birkh\"auser},
\blocation{Boston, MA}.
\bid{doi={10.1007/978-0-8176-4545-8_7}, mr={2359365}}
\end{bincollection}
%
\bptok{imsref}%
\endbibitem

\bibitem{JPSincomp}
%
\begin{barticle}[mr]
\bauthor{\bsnm{Jarrow},~\bfnm{Robert~A.}\binits{R.~A.}},
\bauthor{\bsnm{Protter},~\bfnm{Philip}\binits{P.}} \AND
\bauthor{\bsnm{Shimbo},~\bfnm{Kazuhiro}\binits{K.}}
(\byear{2010}).
\btitle{Asset price bubbles in incomplete markets}.
\bjournal{Math. Finance}
\bvolume{20}
\bpages{145--185}.
\bid{doi={10.1111/j.1467-9965.2010.00394.x}, issn={0960-1627}, mr={2650245}}
\end{barticle}
%
\bptok{imsref}%
\endbibitem

\bibitem{KN}
%
\begin{bmisc}[auto:STB|2014/08/04|07:23:14]
\bauthor{\bsnm{Kreher},~\bfnm{D.}\binits{D.}} \AND
\bauthor{\bsnm{Nikeghbali},~\bfnm{A.}\binits{A.}}
(\byear{2013}).
\bhowpublished{A new kind of augmentation of filtrations suitable for
a change of probability measure by a strict local martingale. Preprint,
available at
\arxivurl{arXiv:1108.4243v2}.}
\end{bmisc}
%
\bptok{imsref}%
\endbibitem

\bibitem{fromto}
%
\begin{bmisc}[auto:STB|2014/08/04|07:23:14]
\bauthor{\bsnm{Madan},~\bfnm{D.}\binits{D.}},
\bauthor{\bsnm{Roynette},~\bfnm{B.}\binits{B.}} \AND
\bauthor{\bsnm{Yor},~\bfnm{M.}\binits{M.}}
(\byear{2008a}).
\bhowpublished{From Black--Scholes formula, to local times and last
passage times for certain submartingales.
Available at \surl{hal.archives-ouvertes.fr/hal-00261868}.}
\end{bmisc}
%
\bptok{imsref}%
\endbibitem

\bibitem{MRY}
%
\begin{barticle}[auto:STB|2014/08/04|07:23:14]
\bauthor{\bsnm{Madan},~\bfnm{D.}\binits{D.}},
\bauthor{\bsnm{Roynette},~\bfnm{B.}\binits{B.}} \AND
\bauthor{\bsnm{Yor},~\bfnm{M.}\binits{M.}}
(\byear{2008b}).
\btitle{Option prices as probabilities}.
\bjournal{Finance Res. Lett.}
\bvolume{5}
\bpages{79--87}.
\end{barticle}
%
\bptok{imsref}%
\endbibitem

\bibitem{MadanYor}
%
\begin{bincollection}[mr]
\bauthor{\bsnm{Madan},~\bfnm{Dilip~B.}\binits{D.~B.}} \AND
\bauthor{\bsnm{Yor},~\bfnm{Marc}\binits{M.}}
(\byear{2006}).
\btitle{Ito's integrated formula for strict local martingales}.
In \bbooktitle{In Memoriam {P}aul-{A}ndr\'e {M}eyer: S\'eminaire de
{P}robabilit\'es {XXXIX}}.
\bseries{Lecture Notes in Math.}
\bvolume{1874}
\bpages{157--170}.
\bpublisher{Springer},
\blocation{Berlin}.
\bid{doi={10.1007/978-3-540-35513-7_13}, mr={2276895}}
\end{bincollection}
%
\bptok{imsref}%
\endbibitem

\bibitem{Meyer}
%
\begin{bincollection}[mr]
\bauthor{\bsnm{Meyer},~\bfnm{P.~A.}\binits{P.~A.}}
(\byear{1972}).
\btitle{La mesure de H. F\"ollmer en th\'eorie des surmartingales}.
In \bbooktitle{S\'eminaire de {P}robabilit\'es, VI ({U}niv.
{S}trasbourg, Ann\'ee Universitaire 1970--1971; {J}ourn\'ees
{P}robabilistes de {S}trasbourg, 1971)}
\bpages{118--129}.
\bseries{Lecture Notes in Math.}
\bvolume{258}.
\bpublisher{Springer},
\blocation{Berlin}.
\bid{mr={0368131}}
\end{bincollection}
%
\bptok{imsref}%
\endbibitem

\bibitem{Urusov}
%
\begin{barticle}[mr]
\bauthor{\bsnm{Mijatovi{\'c}},~\bfnm{Aleksandar}\binits{A.}} \AND
\bauthor{\bsnm{Urusov},~\bfnm{Mikhail}\binits{M.}}
(\byear{2012}).
\btitle{On the martingale property of certain local martingales}.
\bjournal{Probab. Theory Related Fields}
\bvolume{152}
\bpages{1--30}.
\bid{doi={10.1007/s00440-010-0314-7}, issn={0178-8051}, mr={2875751}}
\end{barticle}
%
\bptok{imsref}%
\endbibitem

\bibitem{newkind}
%
\begin{barticle}[mr]
\bauthor{\bsnm{Najnudel},~\bfnm{Joseph}\binits{J.}} \AND
\bauthor{\bsnm{Nikeghbali},~\bfnm{Ashkan}\binits{A.}}
(\byear{2011}).
\btitle{A new kind of augmentation of filtrations}.
\bjournal{ESAIM Probab. Stat.}
\bvolume{15}
\bpages{S39--S57}.
\bid{doi={10.1051/ps/2010026}, issn={1292-8100}, mr={2817344}}
\end{barticle}
%
\bptok{imsref}%
\endbibitem

\bibitem{PalP}
%
\begin{barticle}[mr]
\bauthor{\bsnm{Pal},~\bfnm{Soumik}\binits{S.}} \AND
\bauthor{\bsnm{Protter},~\bfnm{Philip}\binits{P.}}
(\byear{2010}).
\btitle{Analysis of continuous strict local martingales via {$h$}-transforms}.
\bjournal{Stochastic Process. Appl.}
\bvolume{120}
\bpages{1424--1443}.
\bid{doi={10.1016/j.spa.2010.04.004}, issn={0304-4149}, mr={2653260}}
\end{barticle}
%
\bptok{imsref}%
\endbibitem

\bibitem{para}
%
\begin{bbook}[mr]
\bauthor{\bsnm{Parthasarathy},~\bfnm{K.~R.}\binits{K.~R.}}
(\byear{1967}).
\btitle{Probability Measures on Metric Spaces}.
\bseries{Probability and Mathematical Statistics}
\bvolume{3}.
\bpublisher{Academic Press},
\blocation{New York}.
\bid{mr={0226684}}
\end{bbook}
%
\bptok{imsref}%
\endbibitem

\bibitem{benchmark}
%
\begin{bbook}[mr]
\bauthor{\bsnm{Platen},~\bfnm{Eckhard}\binits{E.}} \AND
\bauthor{\bsnm{Heath},~\bfnm{David}\binits{D.}}
(\byear{2006}).
\btitle{A Benchmark Approach to Quantitative Finance}.
\bpublisher{Springer},
\blocation{Berlin}.
\bid{doi={10.1007/978-3-540-47856-0}, mr={2267213}}
\end{bbook}
%
\bptok{imsref}%
\endbibitem

\bibitem{optionpricesprobabilities}
%
\begin{bbook}[mr]
\bauthor{\bsnm{Profeta},~\bfnm{Christophe}\binits{C.}},
\bauthor{\bsnm{Roynette},~\bfnm{Bernard}\binits{B.}} \AND
\bauthor{\bsnm{Yor},~\bfnm{Marc}\binits{M.}}
(\byear{2010}).
\btitle{Option Prices as Probabilities. A New Look at Generalized Black--Scholes
Formulae}.
\bpublisher{Springer},
\blocation{Berlin}.
\bid{doi={10.1007/978-3-642-10395-7}, mr={2582990}}
\end{bbook}
%
\bptok{imsref}%
\endbibitem

\bibitem{Protterjumps}
%
\begin{bmisc}[auto:STB|2014/08/04|07:23:14]
\bauthor{\bsnm{Protter},~\bfnm{Ph.~E.}\binits{Ph.~E.}}
(\byear{2013}).
\bhowpublished{Strict local martingales with jumps. Preprint,
available at
\arxivurl{arXiv:1307.2436v2}.}
\end{bmisc}
%
\bptok{imsref}%
\endbibitem

\bibitem{RevuzYor}
%
\begin{bbook}[mr]
\bauthor{\bsnm{Revuz},~\bfnm{Daniel}\binits{D.}} \AND
\bauthor{\bsnm{Yor},~\bfnm{Marc}\binits{M.}}
(\byear{1999}).
\btitle{Continuous Martingales and {B}rownian Motion},
\bedition{3rd} ed.
\bseries{Grundlehren der Mathematischen Wissenschaften [Fundamental
Principles of Mathematical Sciences]}
\bvolume{293}.
\bpublisher{Springer},
\blocation{Berlin}.
\bid{doi={10.1007/978-3-662-06400-9}, mr={1725357}}
\end{bbook}
%
\bptok{imsref}%
\endbibitem

\bibitem{Ruf}
%
\begin{barticle}[mr]
\bauthor{\bsnm{Ruf},~\bfnm{Johannes}\binits{J.}}
(\byear{2013}).
\btitle{Hedging under arbitrage}.
\bjournal{Math. Finance}
\bvolume{23}
\bpages{297--317}.
\bid{doi={10.1111/j.1467-9965.2011.00502.x}, issn={0960-1627}, mr={3034079}}
\end{barticle}
%
\bptok{imsref}%
\endbibitem

\bibitem{YenYor}
%
\begin{barticle}[mr]
\bauthor{\bsnm{Yen},~\bfnm{Ju-Yi}\binits{J.-Y.}} \AND
\bauthor{\bsnm{Yor},~\bfnm{Marc}\binits{M.}}
(\byear{2011}).
\btitle{Call option prices based on {B}essel processes}.
\bjournal{Methodol. Comput. Appl. Probab.}
\bvolume{13}
\bpages{329--347}.
\bid{doi={10.1007/s11009-009-9151-5}, issn={1387-5841}, mr={2788861}}
\end{barticle}
%
\bptok{imsref}%
\endbibitem
\end{thebibliography}
\end{document}